\documentclass[12pt]{article}
\usepackage{authblk}
\usepackage{fixltx2e}
\usepackage[margin=1.3in]{geometry}

\usepackage{amssymb}
\usepackage{color}
\usepackage{amsmath}
\usepackage{amsthm}
\usepackage{amssymb}
\usepackage{amsmath, amsfonts, bm}
\newtheorem{thm}{Theorem}[section]
\newtheorem{cor}[thm]{Corollary}
\newtheorem{lem}[thm]{Lemma}
\newtheorem{prop}[thm]{Proposition}

\theoremstyle{definition}

\newtheorem{rem}[thm]{Remark}

\newtheorem{example}[thm]{Example}
\theoremstyle{remark}

\newcommand{\T}{\overline{T}}
\newcommand{\TT}{{\mathbb T}}
\newcommand{\Sc}{{\mathbb S}}

\newcommand{\Id}{{\rm Id}}
\newcommand{\Homeo}{\rm Homeo}
\newcommand{\supp}{\rm supp}
\newcommand{\U}{\mathcal U}
\newcommand{\K}{\mathcal K}

\newcommand{\G}{\mathcal G}
\newcommand{\du}{{\rm d}\,}
\newcommand{\Inn}{{\rm Inn}}
\newcommand{\inn}{{\rm inn}}
\newcommand{\Ad}{{\rm Ad}}
\newcommand{\GL}{{\rm GL}}
\newcommand{\SL}{{\rm SL}}
\newcommand{\nc}{{\rm (NC)}}
\newcommand{\Aut}{{\rm Aut}}
\newcommand{\Sub}{{\rm Sub}}

\newcommand{\N}{{\mathbb N}}
\newcommand{\Z}{{\mathbb Z}}
\newcommand{\R}{{\mathbb R}}
\newcommand{\C}{{\mathbb C}}
\newcommand{\Q}{{\mathbb Q}}
\newcommand{\HH}{{\mathbb H}}
\newcommand{\LL}{{\mathbb L}}
\newcommand{\PP}{{\mathbb P}}
\let\ol=\overline
\let\ap=\alpha

\let\mi=\setminus

\title{Distal Actions of Automorphisms \linebreak of Lie Groups $G$ on S\MakeLowercase{ub}$_{G}$}
\author{Riddhi Shah and Alok Kumar Yadav}

\begin{document}
\maketitle
%\title[Distal Actions of Automorphisms on $\Sub_G$]
%{Distal Actions of Automorphisms of Lie Groups $G$ on S\MakeLowercase{ub}$_{G}$}
%\author[R.\ Shah and A.\ K.\ Yadav]{RIDDHI SHAH\\ School of Physical Sciences\\ Jawaharlal Nehru University\\ New Delhi 110067, India\\
%e-mail\textup{: \texttt{riddhi.kausti@gmail.com, rshah@jnu.ac.in}}\\ \\
%\ and\
%{\em ALOK KUMAR YADAV}\\ Department of Mathematical Sciences\\ Indian Institute of Science Education and Research\\ 
%Mohali 140306, Punjab, India\\
%e-mail\textup{: \texttt{alokmath1729@gmail.com}}}

\begin{abstract}
  For a locally compact metrizable group $G$, we study the action of $\Aut(G)$ on $\Sub_G$, the set of closed subgroups of $G$ 
 endowed with the Chabauty topology. Given an automorphism $T$ of $G$, we relate the distality of the $T$-action on $\Sub_G$ 
 with that of the $T$-action on $G$ under a certain condition. If $G$ is a connected Lie group, we characterise the distality of the 
 $T$-action on $\Sub_G$ in terms of compactness of the closed subgroup generated by $T$ in $\Aut(G)$ under certain conditions on 
 the center of $G$ or on $T$ as follows: $G$ has no compact central subgroup of positive dimension or $T$ is unipotent or $T$ is 
 contained in the connected component of the identity in $\Aut(G)$. Moreover, we also show that a connected Lie group $G$ acts 
 distally on $\Sub_G$ if and only if $G$ is either compact or it is isomorphic to a direct product of a compact group and a vector group. 
 All the results on the Lie groups mentioned above hold for the action on $\Sub^a_G$, a subset of $\Sub_G$ consisting of closed 
 abelian subgroups of $G$. \\


\end{abstract}

\section{Introduction} 
Let $X$ be a (Hausdorff) topological space and let  $T$ be a homeomorphism of $X$. The map $T$ is said to be distal if for any 
pair of distinct points $x,y \in {X}$, the closure of the double orbit $\{(T^{n}(x),T^{n}(y))\mid n\in{\Z}\}$ in $X\times X$ does 
not intersect the diagonal, i.e.\ for $x,y\in X$ with $x\neq y$, 
$\overline{\{(T^{n}(x),T^{n}(y))\mid n\in{\Z}\}}\cap\{(d,d)\mid d\in{X}\} =\emptyset$. The notion of distality was introduced by 
David Hilbert (see Ellis \cite{E4}, Moore \cite{M9}) and studied by many in different contexts (see Abels \cite{A1, A2}, 
Furstenberg \cite{F6}, Raja-Shah \cite{RaSh10, RaSh19}, Shah \cite{Sh12} and the references cited therein). For a semigroup $\Gamma$ 
which acts on $X$ by homeomorphisms, i.e.\ $\Gamma\to \Homeo(X)$ is a homomorphism, we say that $\Gamma$ acts 
distally on $X$ if for any two distinct points $x,y\in X$, the closure of $\{(\gamma(x), \gamma(y))\mid \gamma\in\Gamma\}$ does 
not intersect the diagonal. Let $G$ be a  locally compact (Hausdorff) group with the identity $e$ and let $T\in\Aut(G)$, the group of 
automorphisms (homeomorphisms which are homomorphisms) of $G$. Then $T$ is distal (on $G$) if and only if the closure of the 
$T$-orbit $\{T^n(x)\mid n\in \Z\}$ of $x$ does not contain the identity $e$ unless $x=e$. If a group $\Gamma$ acts on $G$ by 
automorphisms, then the $\Gamma$-action on $G$ is distal if and only if $e\not\in\ol{\Gamma(x)}$ unless $x=e$. 

Let $G$ be a locally compact (Hausdorff) group and let $\Sub_G$ be the set of all closed subgroups of $G$ equipped with the Chabauty 
topology (see \cite{CC}). Note that $\Sub_G$ is compact and Hausdorff. Also, it is metrizable if $G$ is second countable \cite{BP}. 
For various groups $G$, the space $\Sub_G$ has been identified and studied extensively, e.g.\ $\Sub_{\R^2}$ is homeomorphic to 
$\Sc^4$, the unit sphere in $\R^5$ \cite{PH}, $\Sub_{\R}$ is homeomorphic to $[0,\infty]$ with a compact topology and 
$\Sub_{\Z}$ is homeomorphic to the subspace $\{0\}\cup\{{\frac{1}{n}}\mid n\in\N\}$ of $[0,1]$ with the usual topology. The space of 
$G$-invariant measures on $\Sub_G$ and the subspace of lattices have also been studied extensively. We refer the reader to 
Bridson et al \cite{BHK}, Abert et al \cite{ABBGNRS} and the references cited therein. 

There is a natural group action of $\Aut(G)$ on $\Sub_G$; namely, $(T,H)\mapsto T(H)$, $T\in \Aut(G)$, $H\in\Sub_G$. 
Since the image of $\Aut(G)$ under this action is a large subclass of homeomorphisms of $\Sub_G$, it would be significant to study the 
dynamics of this special subclass. Here, we would like to study the action of an automorphism of $G$ on $\Sub_G$ in terms of distality. 
 We show that the distality of the $T$-action on $\Sub_G$ implies the distality of the $T$-action on $G$  for a large class of 
 locally compact groups $G$; namely the class of those locally compact Hausdorff first countable (metrizable) 
 groups which do not admit nontrivial compact connected normal subgroups (more generally, see Theorem \ref{Comp})

Conversely, not all distal automorphisms of $G$ act distally on $\Sub_G$; see Example \ref{counter}. In fact  
we know that any unipotent automorphism of a connected Lie group $G$ is distal \cite{A1, A2}, but it does not act distally on $\Sub_G$ 
unless it is trivial (see Theorem \ref{unip}).

Ellis \cite{E4} characterised distal maps on compact spaces. If $X$ is compact and $T$ is a homeomorphism of $X$, then $T$ is 
distal if and only if $E(T)$, the closure of $\{T^n\mid n\in{\N}\}$ in $X^X$ is a group, where $X^X$ is endowed with the product topology. 
Furstenberg \cite{F6} has a characterisation of minimal distal maps on compact metric spaces. Here we show that for a large class of 
connected Lie groups $G$, if an automorphism $T$ acts distally on $\Sub_G$, a stronger result holds; namely, $T$ is contained in a 
compact subgroup of $\Aut(G)$. For a class of connected Lie groups $G$ which do not admit any compact central subgroup of positive 
dimension, we characterise automorphisms $T$ which act distally on $\Sub_G$ in terms of compactness of the closed group generated by 
$T$ in $\Aut(G)$ (more generally see Theorem \ref{main}). For connected Lie groups, our results actually hold under a much weaker 
condition: e.g.\ if $T$ acts distally on $\Sub^a_G$, the space of closed abelian subgroups of $G$. Note that $\Sub^a_G$ is closed in 
$\Sub_G$ and hence compact, and it is invariant under the action of $\Aut(G)$. If $T\in\Aut(G)$ acts distally on $\Sub_G$, then it acts 
distally on $\Sub^a_G$. The converse holds for a large class of connected Lie groups as shown in Theorem \ref{main}. 

We now discuss a class of automorphisms satisfying a certain weaker condition than those acting distally on $\Sub_G$. For this, we introduce a 
subclass (NC) of $\Aut(G)$:  An automorphism $T$ of a locally compact group $G$ is said to belong to $\nc$, if for any discrete (closed) nontrivial 
cyclic subgroup $A$ of $G$, $T^{n_k}(A)\not\to \{e\}$ in $\Sub_G$ for any sequence $\{n_k\}\subset\Z$ such that ${n_k}\to\infty$ or ${n_k}\to -\infty$. 

Since the cyclic groups are the most basic kind of groups and if the action of the group generated by an automorphism on $\Sub_G$ 
does not `contract' any nontrivial closed cyclic group to the trivial group $\{e\}$, we call the class of such automorphisms $\nc$, 
where NC stands for `non-contracting'. Calling such an automorphism $T$ itself ``non-contracting'' could be misleading as it could 
imply that the contraction group of $T$ is trivial.  

For $G=\R^n$ and $T=\alpha\,\Id$ for any $\alpha\in \R\setminus \{0,1,-1\}$, does not belong to $\nc$, as for 
any cyclic group $A$ in $\R^n$, either $T^n(A)\to\{e\}$ or $T^{-n}(A)\to\{e\}$ as $n\to\infty$. All Lie groups and all connected locally 
compact non-compact groups admit nontrivial discrete cyclic subgroups. However, there are some totally disconnected groups which 
do not admit such subgroups, e.g.\ $\Q_p$, $p$ a prime, has no nontrivial closed discrete cyclic subgroups. 

If $T$ acts distally on $\Sub^a_G$, then $T\in \nc$. Some of the results, mainly those about the distal actions of $T$ on $\Sub^a_G$, 
for connected Lie groups $G$ without compact central subgroups of positive dimension, are proven under a weaker assumption 
that $T\in \nc$. All automorphisms contained in compact subgroups of $\Aut(G)$ belong to $\nc$. The converse also holds for a 
large class of connected Lie groups; one of the main results shows that if $G$ is a connected Lie group without any compact 
central subgroup of positive dimension, then $T$ belongs to (NC) if and only if $T$ is contained in a compact subgroup of 
$\Aut(G)$ (see Theorem \ref{main})

Note that the projective space $\R\PP^n$ can be identified with $\LL^n$, the set of lines in $\R^{n+1}$, which is a closed subset of 
$\Sub^a_{\R^{n+1}}$ invariant under the action of $\GL(n+1,\R)$. The action of $\GL(n+1, \R)$ on $\R\PP^n$ has been studied 
extensively by many mathematicians. Shah and Yadav in \cite{SY1} show that for $T\in\SL(n+1,\R)$, if the homeomorphism of the 
unit $n$-sphere $\Sc^n$ corresponding to $T$ is distal, then $T$ is contained in a compact subgroup of $\SL(n+1,\R)$ 
(more generally, see Corollary 5 in \cite{SY1}). This in turn implies that if the action of such $T$ on $\R\PP^n$, and hence on 
$\LL^n$ is distal, then $T$ is contained in a compact subgroup of $\SL(n+1,\R)$. It is easy to show that for $T\in \GL(n+1,\R)$, if $T$ 
belongs to $\nc$, then $T$ acts distally on $\LL^n\cong\R\PP^n$ (this also follows from Theorem \ref{main}). The converse is 
not true as $T=\ap\,\Id$ on $\R^{n+1}$, ($|\ap|\ne 0,1$), acts trivially on $\R\PP^n\cong\LL^n$ but it does not belong to $\nc$, 
as observed above. Therefore, in case of automorphisms $T$ of $G=\R^{n+1}$, the statements in Theorem \ref{main} are 
not equivalent to the statement that $T$ acts distally on a smaller non-discrete compact $T$-invariant subset $\LL^n\subset \Sub^a_G$. 

In Theorem \ref{main}, the condition on the center of $G$ is necessary as illustrated by the example in Remark \ref{rem1}. Moreover, 
Example \ref{ex} shows that for a certain group $G$, there are nontrivial automorphisms contained in a connected unipotent 
subgroup of $(\Aut(G))^0$ which belong to (NC), where $(\Aut(G))^0$ is the connected component of the identity in $\Aut(G)$ 
with respect to the compact-open topology. Therefore, the condition on the center in Theorem \ref{main} is also necessary even if we restrict 
to the subclass of unipotent automorphisms or those  belonging to $(\Aut(G))^0$.  However for these subclasses, we show that a part of 
Theorem \ref{main} can be generalised to any connected Lie group $G$, i.e.\ $T\in (\Aut(G))^0$ acts distally on $\Sub^a_G$ if 
and only if $T$ is contained in a compact subgroup of $\Aut(G)$  (see Theorem \ref{aut0}). We also show that for a unipotent 
automorphism $T$ of $G$, $T$ acts distally on $\Sub^a_G$ if and only if $T=\Id$, the identity map of $G$ 
(see Theorem \ref{unip}). 

The group $G$ acts on $\Sub_G$ via $\Inn(G)$, the group of inner automorphisms of $G$. We study this action and  characterise the
class of connected Lie groups $G$ which act distally on $\Sub^a_G$. Namely, we show that $G$ acts distally on $\Sub^a_G$ if and only if 
every inner automorphism of $G$ acts distally on $\Sub^a_G$, and that the latter statement is equivalent to the statement that $G$ is
either compact or $G=\R^n\times K$, where $K$ is a compact group and $n\in\N$ (more generally, see Corollary \ref{pt-distal}).

Baik and Clavier in \cite{BC2} show that for $G={\rm PSL}(2, \C)$, a specific subspace of $\Sub^a_G$ is homeomorphic to the one-point compactification of 
$\Sc^2\times\R^4$. They also describe the space which is the closure of all cyclic subgroups of $G$, where $G={\rm PSL}(2, \R)$ 
or $G={\rm PSL}(2, \C)$ in \cite{BC1} and \cite{BC2} respectively. Bridson, de la Harpe and Kleptsyn in \cite{BHK} describe the structure 
of $\Sub^a_\HH$ and various other subspaces of $\Sub_\HH$, for the 3-dimensional Heisenberg group $\HH$, and they also study 
and describe the action of $\Aut(G)$ on some of these spaces in detail. One can correlate the image of $\Aut(G)$ to a subclass of 
homeomorphisms of these spaces, and our results imply in particular that such a homeomorphism is distal if and only if it generates 
a compact group (as a closed group) in $\Aut(G)$, for a large class of connected Lie groups $G$ which include those mentioned above 
from \cite{BC1, BC2, BHK}. 

Throughout, we will assume that all groups are locally compact and Hausdorff. We will often assume that they are second countable. 
A locally compact Hausdorff group is second countable if and only if it is first countable (metrizable) 
and $\sigma$-compact. In particular, a locally compact Hausdorff first countable group is second countable if it is compactly 
generated or, more generally, its quotient modulo any open almost connected subgroup is countable.
 A closed subgroup of a connected Lie group is locally compact, Hausdorff and second countable. 
 Let $G^0$ denote the connected component of the identity $e$ in $G$. It is a closed (normal) characteristic subgroup of $G$. 

\section{Actions of automorphisms of $G$ on {\rm Sub}$_G$}

In this section we discuss basic properties of the Chabauty topology on $\Sub_G$, and discuss certain elementary results about
convergence of sequences in $\Sub_G$ and the action of $\Aut(G)$ on $\Sub_G$. 

Let $G$ be a locally compact (Hausdorff) group. A sub-basis of the Chabauty topology on $\Sub_G$ is given by the sets of the form 
$\mathcal{O}_1(K)=\{A\in \Sub_G\mid A\cap K=\emptyset\}$ and $\mathcal{O}_2(U)=\{A\in \Sub_G\mid A\cap U\neq\emptyset\}$, 
where $K$ is a compact and $U$ is an open subset of $G$. As observed earlier, $\Sub_G$ is compact and Hausdorff, and if $G$ is
second countable, then both $G$ and $\Sub_G$ are metrizable. For more details on the Chabauty topology, see \cite{BHK, PH}. 

Recall that $\Aut(G)$ is the group of all automorphisms of $G$. There is a natural group action of $\Aut(G)$ on $\Sub_G$ 
defined as follows:
$$
\Aut(G)\times\Sub_G\to\Sub_G,\ (T,H)\mapsto T(H);\ T\in\Aut(G),\ H\in\Sub_G.$$ 
 It is easy to see that the map $H\mapsto T(H)$ defines a homeomorphism of $\Sub_G$ for every $T\in\Aut(G)$ 
 (see e.g.\ Proposition 2.1 in \cite{HK}), and the 
 corresponding map $\Aut(G)\to \Homeo(\Sub_G)$ is a group homomorphism. 
 
 All our groups $G$ are locally compact (Hausdorff) and we now assume that they are second countable, 
 which ensures that $\Sub_G$ is metrizable and also that $G$ is metrizable. 

We first state a criterion (cf.\ \cite{BP}, page 161) for the convergence of a sequence in $\Sub_G$ in the following lemma, which 
will be used often.

\begin{lem}\label{Con}
Let $G$ be as above. A sequence $\{H_n\}\subset \Sub_G$ converges to $H\in \Sub_G$ 
if and only if the following hold:
\begin{enumerate}
\item[${\rm(I)}$] For $g\in G$, if there exists a subsequence $\{H_{n_k}\}$ of $\{H_n\}$ with $h_k\in H_{n_k}$, $k\in\N$, such that 
$h_k\to g$ in $G$, then $g\in H$. 
\item[${\rm(II)}$] For every $h\in H$, there exists a sequence $\{h_n\}_{n\in\N}$ such that $h_n\in H_n$, $n\in\N$, and $h_n\to h$.
\end{enumerate}
\end{lem}

The following lemma lists some elementary facts about the convergence in $\Sub_G$ and also the interplay of the action of 
$\Aut(G)$ and the convergence.

\begin{lem}\label{H_n}
Let $G$ be as above. The following statements hold:
\begin{enumerate}
\item[$(1)$] Let $H_n, L_n\in\Sub_G$ be such that $H_n\subset L_n$, $n\in\N$, $H_n\to H$ and $L_n\to L$ in $\Sub_G$. Then $H\subset L$.
\item[$(2)$] Let $H, L_n\in\Sub_G$ be such that $H\subset L_n$ for all $n\in\N$, and $L_n\to L$ in $\Sub_G$. Then $H\subset L$.
\item[$(3)$] Let $H_0,L_0\in\Sub_G$ be such that $H_0\subset L_0$ and $T_n(H_0)\to H$ and $T_n(L_0)\to L$, for some 
$\{T_n\}\subset \Aut(G)$. Then $H\subset L$.
\item[$(4)$] Let $T\in\Aut(G)$ and let $H, L\in\Sub_G$ be such that $T(H)=H$ and $H\subset L$. If $L'$ is any limit point of 
$\{T^n(L)\}_{n\in\Z}$, then $H\subset L'$.
\end{enumerate}
\end{lem}
\begin{proof}
(1) Let $h\in H$. As $H_n\to H$, Lemma \ref{Con}\,(II) implies that there exists a sequence $\{h_n\}$ such that $h_n\in H_n$ for all $n$, 
and $h_n\to h$. Since $H_n\subset L_n$, we have that $h_n\in L_n$, for all $n$. As $L_n\to L$, Lemma \ref{Con}\,(I) implies that 
$h\in L$. Hence $H\subset L$.\\
$(1)\Rightarrow(2)$ as we can put $H_n=H$ in (1), for all $n$.\\
$(1)\Rightarrow (3)$ as we can take $H_n=T_n(H_0)$ and $L_n=T_n(L_0)$ in (1), for all $n$.\\
$(2)\Rightarrow (4)$ is obvious. 
\end{proof}

The following elementary result about the relation between the convergence of a sequence of subgroups and that of the corresponding 
sequence in the quotient group will be useful. 

\begin{lem}\label{Seq} For $G$ as above and a closed normal subgroup $H$ of $G$, let $\pi:G\to G/H$ be the cannonical projection. 
Suppose $\{L_n\}\subset\Sub_G$ is such that $H\subset\cap_{n\in\N}L_n$. Then the following hold:
\begin{enumerate}
\item[$(1)$] If $L_n\to L$ in $\Sub_G$, then $\pi(L_n)\to\pi(L)$.
\item[$(2)$] If $\pi(L_n)\to L'$, then $L_n\to \pi^{-1}(L')$.
\end{enumerate}
\end{lem}

\begin{proof}
(1) Suppose $L_n\to L$. Observe that $\{\pi(L_n)\}$ is relatively compact as $\Sub_{G/H}$ is compact. Let $L'$ be 
a limit point of $\{\pi(L_n)\}$. As $\Sub_{G/H}$ is metrizable, there exists $\{n_k\}\subset\N$ such that $\pi(L_{n_k})\to L'$. 
First we show that $\pi(L)\subset L'$. By Lemma \ref{Con}\,(II), for any $x\in L$, there exists $\{x_n\}$ such that 
$x_n\in L_n$, $n\in\N$, and $x_n\to x$ in $G$. This implies that $\pi(x_n)\to\pi(x)$ (as $\pi$ is continuous). 
Therefore, $\pi(x)\in L'$, and hence $\pi(L)\subset L'$.

Now we show that $L'\subset\pi(L)$. Suppose $x'\in L'$. Again by Lemma \ref{Con}\,(II), there exists  
$\{x_k'\}$ such that $x'_k\in\pi(L_{n_k})$, $k\in\N$, and $x'_k\to x'$ in $G/H$. For $k\in\N$, let $x_k\in L_{n_k}$ be such that 
$\pi(x_k)=x'_k$. There exists a sequence $\{h_k\}\subset H$ such that $x_kh_k\to x$, for some $x\in G$. Now $x_kh_k\in L_{n_k}$ 
as $H\subset L_{n_k}$, for all $k$, and hence $x\in L$ (by Lemma \ref{Con}\,(I)). Moreover, $x'=\pi(x)\in \pi(L)$, and hence $\pi(L)=L'$. 
Since this holds for all limit points of $\{\pi(L_n)\}$, we have that $\pi(L_n)\to\pi(L)$.

(2) Suppose  $\pi(L_n)\to L'$. Here, $\{L_n\}$ is relatively compact as $\Sub_G$ is compact. Let $L\in\Sub_G$ be a limit point of 
$\{L_n\}$. We show that $L=\pi^{-1}(L')$. As $\Sub_G$ is metrizable, there exists a subsequence $\{L_{n_k}\}$ of 
$\{L_n\}$ such that $L_{n_k}\to L$. From (1), we get that $\pi(L)=L'$. As $H\subset L_{n_k}$, for all $k$, $H\subset L$ 
(from Lemma \ref{H_n}\,(2)), which in turn implies that $L=\pi^{-1}(L')$. Since this 
holds for all the limit points $L$ of $\{L_n\}$, it follows that $L_n\to\pi^{-1}(L')$ \end{proof}

Recall that for a locally compact (Hausdorff) group $G$, if $\Aut(G)$ is endowed with the compact-open topology, the map 
$\Aut(G)\times G\to G$, $(T,x)\mapsto T(x)$, $T\in \Aut(G)$, $x\in G$, is continuous. If $\Aut(G)$ is endowed with the modified 
compact-open topology (which is finer than the compact-open topology), then $\Aut(G)$ is a topological group \cite[9.17]{St06}, 
and we show that the action of the group $\Aut(G)$ on $\Sub_G$ is continuous. Note that for a connected Lie group $G$ with 
the Lie algebra $\G$, $\Aut(G)$ is a Lie group as a closed subgroup of $\GL(\G)$ and, the topology on $\Aut(G)$ inherited from 
$\GL(\G)$ coincides with the compact-open topology, as well as with the modified compact-open topology \cite{Ar46, Ho52}. 
 
\begin{lem} \label{mco} Let $G$ be a locally compact Hausdorff group and let $\Aut(G)$ be endowed with the modified compact-open 
topology. Then the map $\Aut(G)\times\Sub_G\to\Sub_G$, $(T,H)\mapsto T(H)$, $T\in\Aut(G)$, $H\in \Sub(H)$, is continuous. 
\end{lem}

\begin{proof} For a compact set $C$ and an open set $U$ in $G$, let $\lfloor C,U\rfloor=\{\phi\in\Aut(G)\mid \phi(C)\subset U\}$ 
and $\lfloor C,U\rfloor^{-1}=\{\phi\in\Aut(G)\mid \phi^{-1}\in \lfloor C,U\rfloor\}$. They are open in $\Aut(G)$ and the collection of all 
such sets form a sub-basis for the modified compact-open topology on $\Aut(G)$ \cite{St06}. 

Let $T\in\Aut(G)$ and $H\in \Sub_G$. Then $T(H)$ belongs to either $O_1(K)$ or $O_2(U)$ for some compact set $K$ or 
an open set $U$ in $G$, where $O_1(K)$ and $O_2(U)$ are open sets in the sub-basis of the topology on $\Sub_G$. 
Suppose $T(H)\in O_1(K)=\{L\in\Sub_G\mid L\cap K=\emptyset\}$. Then $H\cap T^{-1}(K)=\emptyset$. As $H$ is closed, $T^{-1}(K)$ is 
compact and since $G$ is locally compact, there exists an open relatively compact set $W$ containing $T^{-1}(K)$ such that 
$H\cap \ol{W}=\emptyset$. Then $\lfloor K,W\rfloor^{-1}\times O_1(\ol{W})$ is open in $\Aut(G)\times \Sub_G$, it contains $(T,H)$ and if 
$\phi\in \lfloor K,W\rfloor^{-1}$ and $L\in O_1(\ol{W})$, then $\phi(L)\in O_1(K)$.

Now suppose $T(H)\in O_2(U)=\{L\in\Sub_G\mid L\cap U\ne\emptyset\}$. Since $T(H)\cap U\ne\emptyset$, there exists $x\in H$ 
such that $T(x)\in U$. Since $G$ is locally compact and $U$ is open, there exists an open relatively compact set $W_1$ such that 
$T(x)\in W_1\subset \ol{W}_1\subset U$. Let $V=T^{-1}(W_1)$. Then $V$ is relatively compact. As $x\in V$ and $T(\ol{V})\subset U$, 
we get that $(T,H)\in \lfloor\ol{V},U\rfloor\times O_2(V)$ which is open in $\Aut(G)\times\Sub_G$. Moreover, if 
$\phi\in \lfloor\ol{V},U\rfloor$ and $L\in O_2(V)$, then $\phi(L)\cap U\ne\emptyset$. Therefore, the map 
$\Aut(G)\times\Sub_G\to\Sub_G$ as above is continuous. 
\end{proof}

\section{Distality of automorphisms on $G$ and {\rm Sub}$_G$}

In this section, for $T\in\Aut(G)$, we show that the distality of the $T$-action on $\Sub_G$ implies the distality of the corresponding 
action on $\Sub_{G/H}$ for any closed normal $T$-invariant subgroup $H$ of $G$. We also  compare the distality of the $T$-action on 
$\Sub_G$ and that of $T$ on $G$. Theorem 1.3 in \cite{Sh12} shows that an automorphism $T$ of $G$ is distal if and only if for any 
closed normal $T$-invariant subgroup $H$, $T|_H$ is distal and $T$ acts distally on $G/H$. Lemma \ref{Chab} and Example \ref{counter} 
illustrate that only a partial analogue of this theorem holds for the action of automorphisms on $\Sub_G$. We prove some results 
about the distality of automorphisms on a connected Lie group which belong to $\nc$. We also prove a result on the structure 
of a nilpotent group admitting a unipotent automorphism which is a generalisation of Kolchin's Theorem for vector spaces; this will be 
useful to prove some of the main results. We end the section with a useful lemma about a criterion for the behaviour of an automorphism 
which implies that it does not belong to $\nc$. 

\begin{lem} \label{Chab}
Let $G$ be a locally compact Hausdorff group, $T\in\Aut(G)$ and let $H$ be a closed normal $T$-invariant subgroup of $G$. 
Let $\T\in\Aut(G/H)$ be the corresponding map defined as $\T(gH)=T(g)H$, for all $g\in G$. If $T$ acts distally on $\Sub_G$, 
then $T$ acts distally on $\Sub_H$ and $\T$ acts distally on $\Sub_{G/H}$. 
\end{lem}

\begin{proof} 
Suppose $T$ acts distally on $\Sub_G$. As $\Sub_H$ is closed and $T$-invariant, it is easy to see that 
$T$ acts distally on $\Sub_H$. Now we show that $\T$ acts distally on 
$\Sub_{G/H}$. Suppose $G$ is second countable. 
For $i=1,2$, suppose $H_i\in\Sub_{G/H}$ and a sequence $\{n_k\}\subset\Z$ are such that $\T^{n_k}(H_i)\to L$ 
in $\Sub_{G/H}$. Then $T(H)=H\subset \pi^{-1}(H_i)$. By Lemma \ref{Seq}\,(2), 
$T^{n_k}(\pi^{-1}(H_i))\to\pi^{-1}(L)$, for $i=1,2$. As $T$ is distal on $\Sub_G$, $\pi^{-1}(H_1)=\pi^{-1}(H_2)$, and 
hence $H_1=H_2$. Thus, $\T$ acts distally on $\Sub_{G/H}$ in this case.

Now suppose $G$ is not second countable.
Let $S_H:=\{L\in\Sub_G\mid H\subset L\}$. As $T(H)=H$, it follows that $S_H$ is a closed $T$-invariant subspace of $\Sub_G$. Let 
$\bar\pi:S_H\to\Sub_{G/H}$, $\bar\pi(L)=\pi(L)=L/H$, $L\in S_H$. Then $\bar\pi$ is a homeomorphism such that 
$\T\circ\bar\pi=\bar\pi\circ T$. As $T$ acts distally on $\Sub_G$, it acts distally on $S_H$. Now using the homeomorphism $\bar\pi$, 
it is easy to see that $\T$ acts distally on $\Sub_{G/H}$. 
\end{proof}

The converse of the lemma does not hold as illustrated by the following.

\begin{example} \label{counter}
Let $T= \begin{bmatrix} 1 & 1\\ 0 & 1  \end{bmatrix}$ in $\GL(2,\R)$. It is clear that the $T$-action on $\R^2$ is distal. Also, for 
$H=\R\times\{0\}$, $T$ acts trivially on $H$ and on $\R^2/H$, hence it acts trivially on both $\Sub_H$ and $\Sub_{\R^2/H}$, but 
the $T$-action on $\Sub_{\R^2}$ is not distal, as $T(H)=H$ and for $H_1=\{0\}\times\R$, $T^n(H_1)\rightarrow H$ in $\Sub_{\R^2}$ 
as $n\to\pm\infty$.
\end{example}

Recall that $T\in\Aut(G)$ belongs to $\nc$ if for any discrete (closed) nontrivial cyclic subgroup $A$ of $G$, $T^{n_k}(A)\not\to \{e\}$ 
in $\Sub_G$ for any sequence $\{n_k\}\subset\Z$ such that ${n_k}\to\infty$ or ${n_k}\to -\infty$. 
Note that for $T$ as in Example \ref{counter} and $G=\R^2$, $T\notin\nc$. This is because one can choose a discrete cyclic 
subgroup $A=\{0\}\times\Z$, for which $T^n(A)=\{(nz,z)\mid z\in \Z\}$, $n\in\Z$, and $T^n(A)\to\{(0,0)\}$ as $n\to\pm\infty$.

A topological group is said to be monothetic if it has a dense cyclic subgroup. Monothetic groups are abelian. In a locally compact 
group, a closed monothetic subgroup is either compact or a  discrete infinite cyclic group. The following lemma will be useful. 

\begin{lem} \label{abelian}
Let $G$ be a locally compact Hausdorff second countable group and let $T\in \Aut(G)$. Let $Z$ be a closed central 
$T$-invariant subgroup of $G$ and let $\T$ be the corresponding automorphism on $G/Z$. Suppose $T$ acts distally on $\Sub^a_G$. 
Then $\T$ satisfies the following: For any two closed monothetic subgroups $A_1$ and $A_2$ in $G/Z$, if for a sequence 
$\{n_k\}\subset\Z$, $\T^{n_k}(A_i)\to B$, $i=1,2$, for some closed subgroup $B$ in $G/Z$, then $A_1=A_2$. In particular, 
$\T\in \nc$ in $\Aut(G/Z)$.
\end{lem} 

\begin{proof} Let $\pi:G\to G/Z$ be the natural projection. Since each $A_i$ is monothetic and $Z$ is a central subgroup of $G$, 
it is easy to see that each $\pi^{-1}(A_i)$ is abelian. By Lemma \ref{Seq}\,(2), $T^{n_k}(\pi^{-1}(A_i))\to\pi^{-1}(B)$, $i=1,2$. 
Therefore, $\pi^{-1}(B)\in \Sub^a_G$. As $T$ acts distally on $\Sub_G$, we get that $\pi^{-1}(A_1)=\pi^{-1}(A_2)$ and hence that  
$A_1=A_2$. The second assertion follows from the first. 
\end{proof}

The following useful lemma can be proved easily using Lemma \ref{mco}. The lemma will apply in particular to the case of a connected 
Lie group $G$ (with the Lie algebra $\G$) as $\Aut(G)$ is a Lie group whose topology (inherited from $\GL(\G)$) coincides with the 
compact-open topology, as well as with the modified compact-open topology. 

\begin{lem} \label{cpt-star} Let $G$ be a locally compact Hausdorff group and let $T\in\Aut(G)$ be such that $T=S\phi=\phi S$ for some 
$\phi,S\in \Aut(G)$, where $\phi$ is contained in a compact subgroup of $\Aut(G)$ with respect to the modified compact-open topology. 
Then $T\in\nc$ if and only if $S\in\nc$. 
\end{lem}

 For a locally compact group $G$, $T\in\Aut(G)$ and a compact $T$-invariant subgroup $K$ of $G$, 
$C_K(T)=\{x\in G\mid T^n(x)K\to K\mbox{ in } G/K\mbox{ as } n\to\infty\}$ is known as the $K$-contraction group of $T$. 
The group $C(T)=C_{\{e\}}(T)$ is known as the contraction group of $T$. It has been shown in \cite{RaSh19} that $T$ is distal 
if and only if both $C(T)$ and $C(T^{-1})$ are trivial. For a connected Lie group $G$, it is well-known that 
$C_K(T)=C(T)K$  \cite{HS}; see \cite{RaSh19} for a more general result on this decomposition. 

For an almost connected locally compact group $G$, let $K$ be the largest compact normal subgroup. Then $K$ is characteristic in $G$ 
and $G/K$ is a Lie group. As observed in \cite{RaSh19}, every inner automorphism of $G$ acts distally on $K$. If any $T\in\Aut(G)$ 
acts distally on $K$, then $C(T)$ is closed \cite[Proposition 4.3]{RaSh19}. The following useful lemma gives a more general result 
in the case of connected Lie groups. Note that the largest compact connected central subgroup of $G$ is characteristic in $G$. 
 
 \begin{lem} \label{central-cpt} Let $G$ be a connected Lie group and let $C$ be the largest compact connected central 
 subgroup of $G$. Let $T\in\Aut(G)$. Then the following hold:
 \begin{enumerate}
 \item[$(1)$] $\ol{C(T)}\subset C(T)C$.
 \item[$(2)$] If $T$ acts distally on $C$, then $C(T)$ is closed. In particular, if $G$ does not have any compact central subgroup of 
 positive dimension, then $C(T)$ is closed. 
 \end{enumerate}
  \end{lem}
 
\begin{proof} (1) Let $D=\{d_n\}$ be a countable subgroup in $C(T)$ which is dense in $\ol{C(T)}$ and let 
$\mu=\sum_{n\in\N} 2^{-n}\delta_{d_n}$ be a probability measure on $G$. It is easy to see that $T^i(\mu)\to\delta_e$ as $i\to\infty$. 
Now we can apply Theorem 1.1 of \cite{DS} and get that $\ol{C(T)}=\supp\mu\subset C_C(T)$. As $G$ is a connected Lie group, 
we have that $C_C(T)=C(T)C$ \cite[Theorem 2.4]{HS}. Therefore, we get that $\ol{C(T)}\subset C(T)C$ and (1) follows. 
Moreover, $C(T)C=\ol{C(T)}C$ and hence $C(T)C$ is closed. 

(2) Now suppose $T$ acts distally on $C$. Then $C(T)\cap C=\{e\}$. Let $G'=C(T)C=\ol{C(T)}C$. Then $G'$ is a closed 
$T$-invariant subgroup of $G$ and $G'=C_C(\tau)$, where $\tau=T|_{G'}$. As $C(T)=C(\tau)$, $C(\tau)\cap C=\{e\}$. 
By Corollary 2.7 of \cite{HS}, we get that $C(\tau)$ is closed in $G'$. As $G'$ is closed in $G$, we get that $C(T)$ is closed. 
In particular, if $G$ does not have any compact central subgroup of positive dimension, i.e.\ if $C$ is trivial, then $C(T)$ is closed. 
\end{proof}

For $T\in\Aut(G)$, if the $T$-action on $G$ is distal, it does not imply that the $T$-action on $\Sub_G$ is distal; see Example \ref{counter}. 
Conversely, the following theorem shows that for a large class of locally compact groups $G$, the distality of the $T$-action on $\Sub_G$ 
implies the distality of the $T$-action on $G$.

\begin{thm}\label{Comp}
Let $G$ be a locally compact first countable $($metrizable$)$ group, $T\in\Aut(G)$ and let $K$ be the largest compact normal subgroup of $G^0$. 
Suppose $T$ acts distally on $\Sub_G$. Then the $T$-action on $G/K^0$ is distal. Moreover, if $T$ acts distally on $K^0$, then 
$T$ acts distally on $G$.
\end{thm}

\begin{proof} 
Let $G$ be as above and let $T\in\Aut(G)$. Suppose $T$ acts distally on $\Sub_G$. 
Let $K$ be as in the hypothesis. Note that $K^0$ is characteristic 
in $G$. In particular, it is $T$-invariant. We want to first show that the $T$-action on $G/K^0$ is distal. 

 Since $T$ acts distally on $\Sub_G$, by  Lemma \ref{Chab}, 
$T$ acts distally on $\Sub_{G/K^0}$. Hence, without loss of any generality, we may assume that $K$ as above is totally disconnected 
and show that $T$ is distal. By Theorem 4.1 of \cite{RaSh19}, $T$ is distal if and only if both $C(T)$ and $C(T^{-1})$ are trivial.

\noindent{\bf Step 1:} We first assume that $G$ is second countable. Suppose $G$ is totally disconnected. 
Then $K=\{e\}$. If possible, suppose that $C(T)$ is nontrivial. Since 
$G$ is totally disconnected, there exists a neighbourhood basis of open compact 
subgroups $\{C_m\}_{m\in\N}$ at the identity $e$ in $G$. Choose $m$ such that $C(T)\not\subset C_m$. Let $H=C_m\cap \ol{C(T)}$. 
As $\Sub_G$ is compact, there exists a sequence $\{n_k\}\subset\N$ such that $T^{-n_k}(H)\to L$ for 
some $L\in\Sub_G$. As $\ol{C(T)}$ is $T$-invariant and $H\subset \ol{C(T)}$, by Lemma \ref{H_n}\,(3) we get that 
$L\subset \ol{C(T)}$. Let $x\in C(T)$. Since $C_m$ is an open subgroup in $G$ and $C(T)$ is $T$-invariant, we get that 
$T^n(x)\in C_m\cap C(T)\subset H$ for all large $n$, and hence, that $x=T^{-n}(T^n(x))\in T^{-n}(H)$ for all large $n$. 
Therefore, $x\in L$ and hence, $C(T)\subset L$. As $L$ is closed, $\overline{C(T)}\subset L$. Thus, we have that $\overline{C(T)}=L$, 
i.e.\  $T^{-n_k}(H)\to\overline{C(T)}$. As $\ol{C(T)}$ is $T$-invariant and $H\ne \ol{C(T)}$, we get that $T$ does not 
act distally on $\Sub_G$, a contradiction. Therefore, $C(T)$ is trivial. Interchanging $T$ and $T^{-1}$ and arguing as above, we can
show that $C(T^{-1})$ is trivial. 

\noindent{\bf Step 2:} Suppose $G$ is not totally disconnected. Note that $G^0$ is $T$-invariant and $G/G^0$ is totally disconnected. 
Let  $\T:G/G^0\to G/G^0$ be the automorphism corresponding to $T$. Then from Lemma \ref{Chab}, $\T$ acts distally on $\Sub_{G/G^0}$. 
Hence, from the assertion in Step 1, $C(\T)$ and $C(\T^{-1})$ are trivial. Therefore, $C(T)\cup C(T^{-1})\subset G^0$. As $K$ is totally 
disconnected and normal in $G^0$, by Lemma 2.2 of \cite{Sh19}, it is central in $G^0$ (this also follows from Theorem $1'$ of \cite{I}). 
 Now $\Sub_K=\Sub^a_K$. From Step 1, if $T$ acts distally on $\Sub^a_K$, we get that $T$ acts distally on $K$. Now it is enough to 
 show that $T$ acts distally on $G/K$. As $T$ acts distally on $\Sub_{G/K}$, replacing $G$ by $G/K$ and $T$ by the corresponding 
 automorphism of $G/K$, we may assume that $G$ is a connected Lie group without any nontrivial compact normal subgroup. We have 
 also assumed that $T$ acts distally on $\Sub_G$. We want to show that $T$ is distal. 
   
\noindent{\bf Step 3:} We prove a more general statement: If $G$ is any connected Lie group, $T\in\Aut(G)$ acts distally on the largest 
compact connected central subgroup of $G$ and if $T\in\nc$, then $T$ is distal. 

Assume that $G$ and $T$ are as above. By Lemma \ref{central-cpt}\,$(2)$, both $C(T)$ and $C(T^{-1})$ are closed, 
and hence, simply connected and nilpotent. 

If possible, suppose $C(T)$ is nontrivial. Let $V=Z(C(T))$, the center of $C(T)$. Then $V\cong\R^m$, for some $m\in\N$. Let 
$T_1=T|_{V}$ be the restriction of $T$ to $V$. Then $T_1\in\GL(V)$ and $C(T_1)=V$, which is turn implies that all the eigenvalues 
of $T_1$ have absolute value less than 1. Suppose $T_1$ has a real eigenvalue $\lambda$. Then $0<|\lambda|<1$ and there exists 
a subspace $W\cong\R$ in $V$ such that $T_1(x)=\lambda x$ for all $x\in W$. For the discrete subgroup $\Z$ in $\R\cong W$, 
it is easy to see that $T_1^{-n}(\Z)=\lambda^{-n}\Z\to\{0\}$ as $n\to\infty$ in $\Sub_\R$ \cite{BC1}. As $T|_W=T_1|_W$, this 
leads to a contradiction as $T\in\nc$. 

Now suppose all the eigenvalues of $T_1$ are complex. There exists a $T_1$-invariant subspace $W'\cong\R^2$ in $V$ such that 
$T_1|_{W'}$ has a complex eigenvalue of the form $r(\cos\theta+i\sin\theta)$, where $r\in\R$ and $0<r<1$. Let 
$T_2\in \GL(2,\R)$ be such that $T|_{W'}=T_1|_{W'}=T_2$ (under the isomorphism of $\R^2$ with $W'$). Then $T_2=rA_\theta$, where 
$A_\theta=AR_\theta A^{-1}$ for some $A\in \GL(2,\R)$, and $R_\theta$ is the rotation by the angle $\theta$ on $\R^2$.
Here, $A_\theta$ generates a relatively compact group in $\GL(2,\R)$. Let $Z=\{(m,0)\mid m\in\Z\}$. Then $Z\in\Sub_{\R^2}$. Since $0<r<1$, 
we have $r ^{-n}Z\to\{(0,0)\}$ as $n\to\infty$ \cite{BC1}, and that $r\,\Id\not\in\nc$. Since $T_2=(r\,\Id)A_\theta=A_\theta(r\,\Id)$,
and $A_\theta$ generates a relatively compact group, by Lemma \ref{cpt-star}, we get that
$T_2\not\in\nc$. As $T_2=T|_{W'}$, this leads to a contradiction. Hence $C(T)$ is trivial. Replacing $T$ by $T^{-1}$ and arguing 
as above, we get that $C(T^{-1})$ is also trivial. Therefore, $T$ is distal. 

We have proved the first assertion in case $G$ is second countable. Now suppose $G$ is not second countable. 
As noted before Step 1, $K$ can be assumed to be totally disconnected and $T$ acts distally on $\Sub_G$. We need to show that $T$ is distal. 
Suppose $x\in G$ is such that $T^{n_k}(x)\to e$ in $G$ for sone sequence $\{n_k\}\subset \Z$. 
Let $L_x$ be the subgroup generated by $O_x=\{T^n(x)\mid n\in\Z\}$ in $G$ and let 
$H=\ol{L_xG^0}$. Then $H$ is a closed $T$-invariant subgroup of $G$. As $O_x$, and hence, $L_x$ is countable, we have that 
$H/H'$ is countable for some open almost connected subgroup $H'$ of $H$.  Since $H'$ is compactly generated and first countable, and hence 
second countable, we get that $H$ is second countable. As $H^0=G^0$, $K$ is the largest compact normal subgroup of $H^0$
and $K$ is totally disconnected. Since $T$ acts distally on $\Sub_G$ and since $\Sub_H$ is closed in $\Sub_G$, 
we have that $T|_H$ acts distally on $\Sub_H$. As $H$ is second countable, we get from above that $T|_H$ is distal. 
Therefore, $x=e$ and $T$ is distal.

We have proved that if $T$ acts distally on $\Sub_G$, then $T$ acts distally on $G/K^0$, where $K$ is the largest compact normal subgroup 
of $G^0$. Moreover, if $T$ also acts distally on $K^0$, then it is easy to show that $T$ acts distally on $G$. 
\end{proof}
  
We get a stronger result in case of Lie groups as follows.

\begin{cor} \label{Lie-distal} Let $G$ be a Lie group with not necessarily finitely many connected components and let 
$T\in\Aut(G)$ and let $C$ be the largest compact connected central 
subgroup of $G^0$. Then the following hold:
\begin{enumerate}
\item[{$(a)$}] If $T$ acts distally on $C$ and $T\in\nc$, then $T$ acts distally on $G$. 
\item[{$(b)$}] If $T$ acts distally on $\Sub^a_G$, then $T$ acts distally on $G/C$.
\end{enumerate}
\end{cor}

\begin{proof}  As $G/G^0$ is discrete and $C$ as above is characteristic in $G$, $T$ acts distally 
on $G$ (resp.\ $G/C$) if and only if $T$ acts distally on $G^0$ (resp.\ $G^0/C$). Moreover, $\Sub^a_{G^0}$ is closed in $\Sub^a_G$ 
and, if $T\in\nc$ (resp.\ $T$ acts distally on $\Sub^a_G$), then $T|_{G^0}\in\nc$ (resp.\ $T|_{G^0}$ acts distally on 
$\Sub^a_{G^0}$). Thus, to prove (a) and (b), we may assume that $G$ is connected and $C$ is central in $G$. 
Now (a) follows from the general statement in Step 3 of the proof of Theorem \ref{Comp}. To prove (b), suppose $T$ acts distally 
on $\Sub^a_G$, where $G$ is connected as assumed above. As $C$ is central in $G$, by Lemma \ref{abelian} we get that the 
automorphism of $G/C$ corresponding to $T$ belongs to $\nc$. As $G/C$ does not have a compact central 
subgroup of positive dimension, (a) implies that $T$ acts distally on $G/C$. 
\end{proof}

\begin{rem} \label{rem1} Note that the first condition in $(a)$ of Corollary 3.7 is necessary as every automorphism of the compact connected abelian
Lie group $G=\TT^n$, ($n\geq 2$), belongs to $\nc$, but $G$ admits automorphisms which are not distal. Here, $\Aut(G)$ is isomorphic to 
$\GL(n,\Z)$. Any closed cyclic group $A$ of $G$ is finite and it is contained in a finite characteristic subgroup $A_m$, the set of all $m$-th 
roots of unity, for some $m\in\N$ which depends on $A$. Now for any $T\in\Aut(G)$, $T|_{A_m}$ has finite order, and hence $T\in\nc$. 
However, if we take any hyperbolic map in $\GL(2,\Z)$, its eigenvalues are of absolute value other than 1, and hence it is not distal on
$\TT^2$. For $n\geq 3$, we can take a hyperbolic map on $\TT^2$ and extend it to $\TT^n\cong\TT^2\times \TT^{n-2}$ naturally by 
putting the identity map on $\TT^{n-2}$.
\end{rem}

A locally compact group $G$ is said to be pointwise distal if every inner automorphism is distal on $G$. The group $G$ is said 
to be distal if the conjugacy action of $G$ on $G$ is distal, i.e.\ for $x\in G$, the closure of 
$\{gxg^{-1}\mid g\in G\}$ does not contain the identity $e$ unless $x=e$. We get the following for inner automorphisms of almost 
connected groups.

\begin{cor}  \label{alm-conn} Let $G$ be an almost connected locally compact Hausdorff group. If every inner 
automorphism of $G$ belongs to $\nc$, then $G$ is distal. In particular, if every inner automorphism of $G$ acts distally 
on $\Sub^a_G$, then $G$ is distal. 
\end{cor}

\begin{proof}
 Let $K$ be the largest compact normal subgroup of $G$. Let $T$ be an inner automorphism of $G$. It is easy to see that $T$ acts 
 distally on $K$ (see the last part of the proof of Theorem 5.2 in \cite{RaSh19}). Hence $C(T)$ is closed and a simply connected nilpotent Lie group 
 \cite[Proposition 4.3]{RaSh19}. In particular, it has no nontrivial compact subgroups. Moreover, if $T\in\nc$, then $T|_{C(T)}$ also 
 belongs to $\nc$, and from Corollary \ref{Lie-distal}\,$(a)$, it follows that $T|_{C(T)}$ is distal  i.e.\ $C(T)=\{e\}$. Since this is true for all inner automorphisms of $G$, by 
 Theorem 4.1 of \cite{RaSh19}, we have that $G$ is pointwise distal, and hence it is distal 
 \cite[Theorem 9]{Ro}. The second assertion follows from the first. 
\end{proof}

Recall that an automorphism $T$ of a connected Lie group is unipotent if the corresponding map $\du T$ on the Lie algebra is a 
unipotent linear transformation, i.e.\ all the eigenvalues of $\du T$ are equal to 1. The following proposition will be very useful.  
In the special case of vector spaces, it is well-known as Kolchin's Theorem. It is also easy to deduce for compact connected 
abelian Lie groups by considering the corresponding map on the Lie algebra (see also Lemma 2.5 in \cite{A2}). 

\begin{prop} \label{kolchin} Let $G$ be a connected nilpotent Lie group and let $T\in\Aut(G)$ be unipotent. Then either $T=\Id$ or 
$G$ has an increasing sequence of closed connected normal $T$-invariant subgroups 
$\{e\}= G_0 \subset G_1\subset \cdots \subset G_n=G$, $n\geq 2$, such that $T$ acts trivially on $G_i/G_{i-1}$, $1\leq i\leq n$, and 
$T$ does not act trivially on $G_{i+1}/G_{i-1}$, $1\leq i\leq n-1$.
\end{prop}

\begin{proof} Suppose $T\ne\Id$. As $G$ is nilpotent, $G/K$ is simply connected, where $K$ is the largest compact 
connected central subgroup in $G$. Then $G$ has a sequence of closed connected normal $T$-invariant subgroups 
$\{e\}=Z_0\subset Z_1=K\subset \cdots \subset Z_m=G$, such that $Z_1/Z_0=K$ and if $G$ is not compact, $Z_i/Z_{i-1}$ is 
a vector group and it is the center of $G/Z_{i-1}$, for all $i=2,\ldots, m$. The automorphism on each $Z_i/Z_{i-1}$ 
corresponding to $T$ is unipotent. From  
Kolchin's Theorem for vector spaces and Lemma 2.5 of \cite{A2} for compact connected abelian Lie groups, it follows that there 
exists a sequence of closed connected normal subgroups in each $Z_i/Z_{i-1}$ such that $T$-acts trivially on each successive 
quotient. Taking the pre-images of these subgroups in $Z_i$, we see that they are closed and $T$-invariant. Also since $Z_i/Z_{i-1}$ 
is central in $G/Z_{i-1}$, any subgroup of $Z_i$ which contains $Z_{i-1}$ is normal in $G$.   Now we have a finite sequence of  
closed connected normal $T$-invariant subgroups $\{e\}=H_0\subset \cdots\subset H_l=G$, such that $T$ acts trivially on 
$H_i/H_{i-1}$, $i=1,\ldots l$. As $T\ne\Id$ and $T$ acts trivially on $H_1$, there exists $i_1$ such that $1\leq i_1< l$, $T$ acts 
trivially on $H_{i_1}$ and $T$ does not act trivially on $H_{i_1+1}$. Let $G_1=H_{i_1}$. Having chosen $G_k=H_{i_k}$ for 
$k\geq 1$, if $G_k\ne G$ (equivalently, $i_k\ne l$), we choose $G_{k+1}=H_j$, where $i_k<j\leq l$ and $j$ is the largest such 
natural number such that $T$ acts trivially on $H_j/G_k$.  As $l$ is finite, there exists $n$ such that $G_n=H_l=G$. 
\end{proof}

The following is known but we give a proof for the sake of completeness. It does not hold for general compact abelian groups as there 
are some compact totally disconnected abelian groups which are torsion-free, e.g.\ $\Z_p$, the ring of $p$-adic integers in $\Q_p$, 
$p$ a prime.

\begin{lem} \label{cpt-roots}
Let $G$ be a nontrivial connected compact abelian Lie group. Let $B_n=\{x\in G\mid x^n=e\}$, $n\in\N$. Then the following holds: 
If $\{a_n\}\subset\N$ is an unbounded sequence, then $\cup_n B_{a_n}$ is dense in $G$. In particular, the torsion group 
$B=\cup_nB_n$ is a dense subgroup in $G$. 
\end{lem}

\begin{proof} 
 Note that each $B_n$ is a finite subgroup and as $G$ is abelian, $B_mB_n\subset B_{mn}$, $m,n\in\N$. Hence $B$ is a group. 
It follows from the first statement that $B$ is dense in $G$. 

Let $B'=\cup_n B_{a_n}$. We may assume that $\{a_n\}$ is strictly increasing and that $a_n\to \infty$. As $G=\TT^k=(\Sc^1)^k$, 
where $\Sc^1$ is the unit circle and $k=\dim G$, it is enough to prove the statement for $G=\Sc^1$. There exists a continuous real 
one-parameter subgroup $\{x(t)\}_{t\in\R}$ such that $x(0)=x(1)=e$, $x(t)\ne e$ for all $t\in ]0,1[$ and $G=\{x(t)\mid t\in[0,1[\}$. 
Now $x(1/a_n)\in B_{a_n}$ and $x(1/a_n)\to x(0)=e$. Let $t\in ]0,1[$.  As $1/a_n\to 0$, there exist $m_n\in\N$, $n\in\N$, such 
that $m_n<a_n$ for all large $n$, and $m_n/a_n\to t$. Therefore, $x(m_n/a_n)\to x(t)$. As $x(m_n/a_n)\in B_{a_n}$, we get that 
$B'$ is dense in $G$. 
\end{proof}

Note that in a locally compact group $G$, the subgroup $G_x$ generated by an element $x$ is either closed and hence a discrete 
cyclic group, or its closure $\ol{G}_x$, being a compactly generated (locally compact) abelian group, is compact. If $G$ is a 
Lie group, then either $G_x$ is discrete or $\ol{G}_x$ is a compact Lie group with finitely many connected components. 

We now state and prove a lemma which will play a crucial role in the proofs of main results in Section 4. Since we need to use 
Lemma~\ref{cpt-roots} for the proof, we state it only for Lie groups even though it may hold for some locally compact 
(non-discrete) groups.

For $T\in\Aut(G)$, let $S_T=\{x\in G\mid T(x)=x\}$ be the stabiliser of $T$. It is a closed subgroup of $G$. 

\begin{lem} \label{discrete}
Let $G$ be a Lie group with not necessarily finitely many connected components and let $T\in\Aut(G)$. 
Let $H$ be a closed normal $T$-invariant subgroup of $G$ such that $H$ does not contain any nontrivial compact subgroup 
and $T|_H=\Id$. Suppose there exists $x\in G$ such that $T(x)=xy$ for some $y\in H$, $y\ne e$.
Let $G_x$ be the cyclic subgroup generated by $x$. If $G_x\cap S_T=\{e\}$, then $T\not\in\nc$.
\end{lem} 

\begin{proof}  Let $G$, $T$, $H$, $x$ and $y$ be as above. Suppose $G_x\cap S_T=\{e\}$. Let $G(x,y)$ be the subgroup of $G$ generated 
by $x$ and $y$. Then $G(x,y)$ is countable and $T$-invariant, and its closure $\ol{G(x,y)}$ is a $T$-invariant Lie subgroup, 
which is separable, and hence second countable. In particular, $\ol{G(x,y)}$ has countably many connected components. Replacing $G$ by 
$\ol{G(x,y)}$, $T$ by its restriction to ${\ol{G(x,y)}}$, and $H$ by $\ol{G(x,y)}\cap H$, we may assume that $G$ is second countable and 
$G/H$ is an abelian Lie group which is either discrete or compact. Now $\Sub_G$ is metrizable. We show that $T\not\in\nc$. 
Since $H$ is $T$-invariant and normal in $G$, we get for every $n\in\Z$ that $T(x^n)=x^ny_n$ for some $y_n\in H$, where $y_n\ne e$ 
unless $x^n=e$, as $G_x\cap S_T=\{e\}$.

\noindent{\bf Step 1:}  Suppose $G/H$ is discrete, i.e.\ $G=G_xH$.  We first show that $G_x$ is discrete. 
Observe that if $G_xH/H$ is finite, then for some $n\in\N$, $x^n\in G_x\cap H\subset G_x\cap S_T=\{e\}$, and hence $G_x$ is finite. 
If possible, suppose $G_x$ is not discrete.  Then $G_x$ is infinite and relatively compact. This implies that $G/H=G_xH/H$ is compact, 
and hence it is finite. Then we get as above that $G_x$ is finite, and it leads to a contradiction. 
Therefore, $G_x$ is discrete, and hence, closed. 

There exists an unbounded sequence $\{n_k\}\subset\N$ such that $T^{n_k}(G_x)\to L$ (say). We show that $L=\{e\}$. 
If possible, suppose $a\in L$ is such that $a\ne e$.  Replacing $a$ by $a^{-1}$ if necessary, we get that there exists 
$\{m_k\}\subset\N$ such that $T^{n_k}(x^{m_k})\to a$. If $G_x$ is finite, then passing to a subsequence if necessary, we can 
choose $m_k=m_1$ for all $k$. Suppose $G_x$ is infinite. As observed above, $G/H$ is also infinite. Let $\pi:G\to G/H$ be 
the natural projection. Since $T(x)\in xH$, we have that $\pi(x^{m_k})\to \pi(a)$. As $G/H$ is infinite and discrete, we get that $\{m_k\}$ is eventually constant. In either case, we have that $T^{n_k}(x^m)\to a$ for some $m\in\N$. Then 
$T^{n_k}(x^m)=x^my_m^{n_k}\to a$. Also, since $a\ne e$, we get that $x^m\ne e$ and hence $y_m\ne e$. Therefore, 
$\{y_m^{n_k}\}_{k\in\N}$ is convergent, and hence $y_m$ is contained in a nontrivial compact subgroup in $H$. 
This leads to a contradiction. Therefore, $L=\{e\}$ and $T\not\in\nc$. 

\noindent{\bf Step 2:} Now suppose $G/H$ is not discrete. As $G_xH$ is dense in $G$, we have that $G/H$ is compact and  
$G_x$ is infinite. Since $G/H$ is abelian and $H\subset S_T$, we have that $S_T$ is normal in 
$G$. Moreover, $G/S_T$ is infinite, as $G_x$ is infinite and $G_x\cap S_T=\{e\}$. As $G$ is a Lie group, $G^0$, and hence, 
$G^0H$ is an open subgroup in $G$. Therefore, $G/G^0H$, being compact, is finite.  We may replace $x$ by $x^n$ 
for some $n\in\N$ and assume that $G=G^0H$. Here, $G^0/(G^0\cap H)$ is connected, and it is also compact as it is isomorphic to $G/H$. 
Since $G/S_T$ is infinite and $H\subset S_T$, we get that $G/S_T=(G^0H)/S_T$, and hence $G^0/(G^0\cap S_T)$ is compact and infinite. 

Note that $G^0$ is $T$-invariant, $T$ acts trivially on $H\cap G^0$ and on $G^0/(H\cap G^0)$. Also, $xh\in G^0$ for some 
$h\in H$, $T(xh)\in xH\cap G^0=xh(H\cap G^0)$ and $G_{xh}\cap S_T=\{e\}$. Moreover $G_{xh}(H\cap G^0)$ is dense in $G^0$ and
$G_{xh}$ is infinite, as $G_{xh}H=G_xH$. 
We may replace $x$ by $xh$ and assume that $x\in G^0$, and we may also replace $G$, $H$ by $G^0$, $H\cap G^0$ 
respectively, and assume $G$ is a connected Lie group, $G/H$ is a compact connected abelian Lie group and $G/S_T$ is infinite and compact. 
Note that $T(g)\in gH$ for all $g\in G$. Let $Q=\{g^{-1}T(g)\mid g\in G\}$. Then $e\in Q\subset H$ and, $Q$ is connected as $G$ is so. 
Therefore, $Q\subset H^0$, and hence $T(g)\in gH^0$ for all $g\in G$. Here, $H/H^0$, being a discrete normal subgroup of $G/H^0$, 
is central in $G/H^0$. Hence, $G/H^0$ is a covering group of $G/H$, and the later is abelian. Therefore, $G/H^0$ is abelian and, 
being connected, it is isomorphic to $\R^n\times \TT^m$, for $n,m\geq 0$ and $m+n\in\N$. Let $\pi_0: G\to G/H^0$ be the natural projection. 

Suppose $m\in\N$. Let $M=\pi_0^{-1}(\TT^m)$. Then $M$ is a closed connected $T$-invariant Lie subgroup and $H^0\subset M$. 
By Theorem 3.7 in Ch.\ XV of \cite{Ho}, $M=CH^0$, where $C$ is a maximal compact connected subgroup of $G$. Now 
$M=C\ltimes H^0$, since $H^0$ does not have any nontrivial compact subgroup. Observe that $C$ is isomorphic to $\pi_0(C)$, 
and hence it is nontrivial and abelian. Suppose $M\not\subset S_T$, i.e.\ $C\cap S_T\subsetneq C$. Note that $C\cap S_T$ is a 
closed (abelian) subgroup of $C$. 

As $C$ is a compact connected abelian Lie group and the set of primes is an unbounded set in $\N$, by Lemma \ref{cpt-roots}, 
the set of elements of prime orders in $C$ is dense in $C$. As $C\setminus (C\cap S_T)$ is open in $C$, there exists 
$b\in C\setminus (C\cap S_T)$ such that $b^p=e$ for some prime $p$. Since $C/(C\cap S_T)$ is also a compact connected 
abelian Lie group, we have that $b^n\in C\setminus (C\cap S_T)$ for all $n<p$ and $T(b^n)=b^nh_n$ for some $h_n\in H$ and 
$h_n\ne e$, $1\leq n<p$. Let $G_b$ be the cyclic subgroup generated by $b$. Then $G_b\cap S_T=\{e\}$. Note that $\pi_0(G_b)$ 
is finite and discrete in $M/H^0$, and hence in $G/H^0$. Now arguing as in Step 1 for $b$, $H^0$ instead of $x$, $H$ respectively, 
we get that $T\not\in\nc$. 

Now suppose $m=0$ or $C\subset S_T$. Then $G/(CH^0)=\R^n$, where $n\in\N$ and, either $C=\{e\}$ or $C\subset S_T$. 
As $H\subset S_T\subsetneq G$ and $G/S_T$, being a compact connected abelian Lie group, is monothetic, 
there exists $x_1\in G\mi S_T$ such that its image in $G/S_T$ generates a (dense) infinite group. Hence 
$G_{x_1}\cap S_T=\{e\}$, where $G_{x_1}$ is the group generated by $x_1$ in $G$. Moreover $\pi_0(x_1)$ generates a discrete 
infinite group in $G/CH^0$, and hence in $G/H^0$. Arguing as in Step 1, for $x_1$, $H^0$ instead of $x$, $H$ respectively, 
we get that $T\not\in\nc$.
\end{proof}

\section{Characterisation of automorphisms of Lie groups $G$ which act distally on {\rm Sub}$_G$}

In this section, we state and prove some of the main results characterising the class of automorphisms which belong to (NC) or act distally on Sub$^a_G$ (see 
Theorems \ref{main}, \ref{unip} and \ref{aut0}). We also prove Proposition \ref{solv} for connected solvable Lie groups which will be 
useful in proving these theorems. After proving Theorems \ref{main}, we prove Theorems \ref{unip} 
and \ref{aut0} which show for all connected Lie groups $G$ that if $T$ is unipotent or if $T\in (\Aut(G))^0$, 
then the following holds: $T$ acts distally on $\Sub^a_G$ if and only if $T$ is contained in a compact subgroup of $\Aut(G)$. 
We also characterise connected Lie groups $G$ which act distally on $\Sub^a_G$ (see Corollary \ref{pt-distal}). 

\begin{thm} \label{main}
Let $G$ be a connected Lie group without any compact central subgroup of positive dimension and let $T\in\Aut(G)$. 
Then the following statements are equivalent.
\begin{enumerate}
\item[$(i)$] $T\in\nc$. 
\item[$(ii)$] $T$ acts distally on $\Sub^a_G$.
\item[$(iii)$] $T$ acts distally on $\Sub_G$.
\item[$(iv)$] $T$ is contained in a compact subgroup of $\Aut(G)$. 
\end{enumerate}
\end{thm}

Note that Theorem \ref{main} does not hold for all connected Lie groups and the condition in the hypothesis is necessary as 
illustrated by the example in Remark \ref{rem1} of the group $G=\TT^n$, $n\geq 2$, for which every $T\in\Aut(G)\cong\GL(n,\Z)$ 
belongs to $\nc$. However, it admits many distal (unipotent) and non-distal automorphisms, each of which generates a discrete infinite 
group in $\Aut(G)$. 

Before proving the theorem, we prove a somewhat stronger result about distality for unipotent automorphisms on connected solvable Lie groups. 
Note that Proposition \ref{solv}\,$(i)$ below will later be generalised to all connected Lie groups, see Theorem \ref{unip}.

\begin{prop} \label {solv}
Let $G$ be a connected solvable Lie group and let $T\in\Aut(G)$ be unipotent.  Then the following hold:
\begin{enumerate}
\item[$(i)$] If $T$ acts distally on $\Sub^a_G$, then $T=\Id$. 
\item[$(ii)$] If $G$ does not have a compact central subgroup of positive dimension and $T\in\nc$, then $T=\Id$. 
\end{enumerate}
\end{prop}

\begin{proof} {\bf Step 1:} Let $G$ be as in the hypothesis. Suppose $T\in\Aut(G)$ acts distally on $\Sub^a_G$. 
Let $K$ be the largest compact connected central subgroup 
of $G$. Then $K$ is characteristic, and in particular, it is $T$-invariant and normal in $G$. If possible, suppose $T|_K\ne\Id$. 
By Proposition \ref{kolchin}, there exist $T$-invariant compact connected (normal) nontrivial subgroups $K_1\subset K_2$ such that 
$T$ acts trivially on $K_1$ and on $K_2/K_1$ but it does not act trivially on $K_2$. We may, without loss of any generality, assume 
that $K_1$ is the largest connected subgroup in $K_2$ such that $T|_{K_1}=\Id$. As $K_1$ and $K_2$ are compact connected 
abelian Lie groups, and $\dim K_1<\dim K_2$, we can get a nontrivial compact connected subgroup $M\subset K_2$ such that 
$M\cap K_1=\{e\}$. As $M\subset K_2$, for every $x\in M$, $T(x)=xy$, for some $y\in K_1$ which depends on $x$. As $M$ is 
monothetic, we can choose $g\in M$, such that $g$ generates a dense subgroup in $M$. Let $h\in K_1$ be such that $T(g)=gh$. 

We now show that $h$ has infinite order. If possible, suppose $h$ has finite order $m$ (say). Then $T(g^{mn})=g^{mn}$, $n\in\N$, and 
as $\{g^{mn}\mid n\in\N\}$ is dense in $M$, we get that $T$ acts trivially on $M$. This leads to a contradiction to our assumption on 
$K_1$ as $K_1\ne MK_1\subset K_2$. As $K_1$ is compact, there exists an unbounded sequence $\{n_k\}\subset\N$ such that 
$h^{n_k}\to h$. Passing to a subsequence if necessary, we may assume that $T^{n_k}(M)\to L$ for some compact group 
$L\subset K_2$. As $\{g^n\}_{n\in\N}$ is dense in $M$ and $T(g)=gh$, we have that $T^{n_k}(M)\subset MK_h$, where $K_h$ is 
the closed subgroup generated by $h$ in $K_1$. 
As $T^{n_k}(g)=gh^{n_k}\to gh$, we get that $gh\in L$. As $M$ is a connected abelian Lie group, there exists a continuous real 
one-parameter group $\{g(t)\}_{t\in\R}\subset M$ such that $g(0)=e$ and $g(1)=g$. Let $x_k=g(1/n_k)\in M$, $k\in\N$. We have that 
$x_k\to e$ and $x_k^{n_k}=g$. Now 
$$
gh=T(g)=T(x_k)^{n_k}=(x_ky_k)^{n_k}=gy_k^{n_k}\mbox{ for some }y_k\in K_1.$$ 
In particular, $h=y_k^{n_k}$ for all $k$, and it implies that $T^{n_k}(gx_k^{-1})=gh^{n_k}h^{-1}x_k^{-1}\to g$. Therefore, 
$g\in L$ and hence $h\in L$.  As $L$ is a closed subgroup, we have that $M\subset L$ and $K_h\subset L$. 
Since $T^{n_k}(M)\subset MK_h$ for all $k$, we have that $L=MK_h$. As $T(MK_h)=MK_h$, this contradicts the hypothesis 
that $T$ acts distally on $\Sub^a_G$. Therefore, $T|_K=\Id$. If $G$ is compact, then $G=K$ and $T=\Id$. 

\noindent{\bf Step 2:} Suppose $G$ is not compact. Let $\pi:G\to G/K$ be the natural projection and let $\T$ be the automorphism 
of $G/K$ corresponding to $T$. Then $\T$ is also unipotent, and by Lemma \ref{abelian}, $\T\in\nc$ in $\Aut(G/K)$. Suppose $(ii)$ 
holds. Since $G/K$ does not have any compact central subgroup of positive dimension, the statement in $(ii)$ implies that $\T=\Id$. 
Now we have that $T|_K=\Id$ and $T(xK)=xK$ for all $x\in G$. Let $N$ be the nilradical of $G$, then $K\subset N$, $N/K$ is simply 
connected and $G/N$ is abelian and isomorphic to $\R^n\times \TT^m$, where $m,n\geq 0$. Let $R$ be the closed connected 
(solvable) normal subgroup containing $N$, such that $R/N=\R^n$, $K\subset N\subset R$ and $R/K$ is a closed connected solvable normal 
subgroup of $G/K$ without any nontrivial compact subgroup. As $G/R=\TT^m$ is compact, we get by Theorem 3.7 in Ch.\ XV of 
\cite{Ho}, that $G=CR$, where $C$ is a maximal compact connected subgroup of $G$ which is abelian. As $K\subset C$, we have 
$T(C)=C$. Here, $T|_C$ is also unipotent and from Step 1, we get that $T|_C=\Id$. 

Now it is enough to show that $T|_R=\Id$. For every $x\in R$, $T(x)\in xK$. There exists a neighbourhood $U$ of the identity $e$ in $G$ 
such that every $x\in U$ is exponential. If $T(x)=x$ for every $x\in U\cap R$, then $T|_R=\Id$, since $U\cap R$ generates $R$. 
If possible, suppose $x\in U\cap R$ is such that $x\not\in K$ and $T(x)\ne x$. Since $x$ is exponential, it can be embedded in a continuous 
real one-parameter subgroup $\{x(t)\}_{t\in\R}$ as $x=x(1)$. We know that $\{x(t)\}_{t\in\R}$ is unbounded (and closed) as $R/K$ has 
no nontrivial compact subgroups. Note that $T(x(t))=x(t)k(t)$, $k(t)\in K$ for all $t\in\R$ and $k(1)=k\ne e$. Let 
$M'=\ol{\{k(t)\mid t\in\R\}}$. It is a connected compact subgroup of $K$. Let $C_x=\{x(t)\}_{t\in\R}$. It is a closed connected 
abelian subgroup of $R$. For some $t_0\in\R$, $k(t_0)$ has infinite order, and hence, there exists an unbounded 
sequence $\{n_m\}\subset\N$ such that $k(t_0)^{n_m}\to k(t_0)$. Passing to a subsequence, we get that $\{k(t_0/l!)^{n_m}\}$ 
converges for every $l\in\N$. As $\Sub^a_G$ is compact, passing to a subsequence if necessary, we get that $\{T^{n_m}(C_x)\}$ 
converges to a closed abelian subgroup $H$ (say). Now $T^{n_m}(x(t_0))=x(t_0)k(t_0)^{n_m}\to x(t_0)k(t_0)$. Therefore,  
 $x(t_0)k(t_0)\in H$. Also, $T^n(x(t/n))=x(t/n)k(t)$, $n\in\N$. Therefore, $T^{n_m}(x(t/n_m))=x(t/n_m)k(t)\to k(t)\in H$ for all $t$, 
 i.e.\ $M'\subset H$ and hence $x(t_0)\in H$. Moreover, $T^{n_m}(x(t_0/l!))=x(t_0/l!)k(t_0/l!)^{n_m}\to x(t_0/l!)h_l$ for some 
 $h_l\in M'$ such that $k(t_0/l!)^{n_m}\to h_l$ for all $l\in\N$. Hence $x(t_0/l!)\in H$ for all $l\in\N$. As $H$ is a closed subgroup, 
 we get that $C_x\subset H$, and hence that $C_xM'\subset H$. Since $C_xM'$ is $T$-invariant, we get by Lemma \ref{H_n}\,(3) that 
 $H=C_xM'$. As $C_x\ne C_xM'$, it contradicts the hypothesis that $T$ acts distally on $\Sub^a_G$. Therefore, $T(x)=x$ for every 
 $x\in U\cap R$ and hence, $T|_R=\Id$. This implies that $T=\Id$ assuming that $(ii)$ holds. Therefore, $(i)$ holds if $(ii)$ holds. 
 
 We now prove $(ii)$. Suppose $G$ has no compact central subgroup of positive dimension. Then its nilradical $N$ is simply 
 connected. Moreover, $\Aut(G)$ is almost algebraic \cite[Theorem 1]{D}, i.e.\ a closed subgroup of a finite index in an algebraic 
 subgroup of $\GL(\G)$, where $\G$ is the Lie algebra of $G$. There exists a continuous unipotent one-parameter subgroup  
 $\{T_t\}_{t\in\R}\subset\Aut(G)$ such that $T=T_1$. Let $G'=\{T_t\}_{t\in\R}\ltimes G$. Then $G'$ is a connected solvable Lie group. 
 By Proposition 3.11 of \cite{Rag}, the commutator subgroup of $G'$ is nilpotent and, as it is contained in $G$, it is contained in $N$. 
 This implies that each $T_t$ acts trivially on $G/N$. 
 
 By Proposition \ref{kolchin}, either $T|_N=\Id$ or there exists a sequence of closed connected normal $T$-invariant subgroups of $N$; 
 $\{e\}=G_0\subset\cdots \subset G_m=N$, $m\geq 2$, such that $T$ acts trivially on $G_k/G_{k-1}$, $1\le k\le m$, and $T$ does not act 
 trivially on $G_{k+1}/G_{k-1}$, $1\le k\le m-1$. 
 
 Suppose $T\in\nc$. If possible, suppose $T|_N\ne \Id$. Note that $G_1\subset G_2\cap S_T\subsetneq G_2$, where $S_T$ 
 is the stabiliser of $T$. There exists $x\in G_2\mi (G_2\cap S_T)$ such that the image of $x$ in $G_2/(G_2\cap S_T)$ 
 generates an infinite group; this is true since $G_2/G_1$ is (simply) connected and nilpotent. Now $G_x\cap S_T=\{e\}$, 
 where $G_x$ is the cyclic subgroup generated by $x$. Moreover, $T(x)=xy$, $y\ne e$ and $y\in G_1$. As $N$ is simply connected 
 and nilpotent, $G_1$ does not contain any nontrivial compact subgroup. By Lemma \ref{discrete}, it follows that 
 $T|_{G_2}\notin\nc$, and hence $T|_N\not\in\nc$, which contradicts the hypothesis. Hence $T|_N=\Id$. If $G$ is nilpotent, then $T=\Id$.
 Now suppose $G$ is not nilpotent.
 
As observed above, $T$ acts trivially on $G/N$. If possible, suppose $T\ne\Id$. Then $N\subset S_T\ne G$ and, as 
$[G,G]\subset N$  \cite[Proposition 3.11]{Rag}, , we get that $S_T$ is a proper closed normal subgroup of $G$ and 
$G/S_T$ is a nontrivial connected abelian group. There exists $x\in G\mi S_T$ such that the image of $x$ in $G/S_T$ generates 
an infinite group. In particular, $G_x\cap S_T=\{e\}$, where $G_x$ is the cyclic group generated by $x$. As $T$ acts trivially on $G/N$, 
$T(x)=xg$ for some $g\in N$ and $g\ne e$. Since $N$ is simply connected and nilpotent, it has no nontrivial compact subgroup. 
Since $T|_N=\Id$, we get by Lemma \ref{discrete} that $T\not\in\nc$.  This contradicts the hypothesis, and hence $T=\Id$. \end{proof}

\begin{proof} [Proof of Theorem \ref{main}.] 
Let $G$ be a connected Lie group such that it does not admit any nontrivial compact connected central subgroup and let $T\in \Aut(G)$. 
It is easy to see that $(iv)\Rightarrow (iii)\Rightarrow (ii)\Rightarrow (i)$. Now suppose $(i)$ holds, i.e.\ $T\in\nc$. To show the equivalence 
of $(i)-(iv)$, it is enough to show that $(iv)$ holds; i.e.\ we need to show that $T$ is contained in a compact subgroup of $\Aut(G)$.  

\noindent{\bf Step 1:} Let $\G$ be the Lie algebra of $G$. For any $\tau\in\Aut(G)$, let 
$\du \tau$ be the corresponding Lie algebra automorphism of $\G$. Then we can view $\Aut(G)$ as a closed subgroup of 
$\GL(\G)$. By Corollary \ref{Lie-distal}\,$(a)$, we have that $T$ acts distally on $G$. By Theorem 1.1 of \cite{A2}, the Lie algebra 
automorphism $\du T$ on the Lie algebra $\G$ of $G$ corresponding to $T$ is distal and hence, all the eigenvalues of $\du T$ 
have absolute value 1 \cite[Theorem $1'$]{A1}. Since $G$ has no compact central subgroup of positive dimension, by 
Theorem 1 of \cite{D} (see also \cite{PW}) $\Aut(G)$ is an almost algebraic subgroup of $\GL(\G)$, i.e.\ a closed subgroup 
of finite index in an algebraic group. Let $H$ be the smallest almost algebraic subgroup containing $T$ in $\Aut(G)$. Then 
$H$ is abelian and it acts distally on $G$ \cite[Corollary 2.3]{A1}, and $H=\K\times\U$, where $\K$ is a compact abelian 
group and $\U$ is unipotent and abelian, and in fact, a unipotent one-parameter group $\{T_t\}_{t\in\R}$ \cite[Corollary 2.5]{A1}. 
Note that all the eigenvalues of each $T_t$ are equal to 1. Then each $T_t$ acts distally on $G$ \cite[Theorem $1'$]{A1}.
Let $T=T_sT_u=T_uT_s$, where $T_s\in \K$ and $T_u=T_{t_0}\in\U$, for some $t_0\in\R$, $t_0\ne 0$. By Lemma \ref{cpt-star}, 
$T_u$, and hence each $T_t$, $t\in\R$, belongs to $\nc$. Now to show that $T$ is contained in a compact 
subgroup of $\Aut(G)$, it enough to show that $T_u=\Id$. 

\noindent{\bf Step 2:} Let $R$ be the solvable radical of $G$. Then each $T_t|_R$ is unipotent and it belongs to $\nc$. Since the largest
 compact connected central subgroup of $R$, being characteristic in $G$, is central in $G$ by Theorem 4 of \cite{I}, and hence it is trivial.
 Therefore, $R$ does not have any compact central subgroup of positive dimension. By Proposition \ref{solv}, $T_t|_R=\Id$. 
 If $G$ is solvable then $T_t=\Id$ for each $t$. 

Suppose $G$ is not solvable. We show that $T_t$ acts trivially on $G/R$. Let $G_s=G/R$. Then  $G_s$ is a connected semisimple 
Lie group. Moreover, $\Inn(G_s)$ is the connected component of the identity in $\Aut(G_s)$ and it is a subgroup of finite index in 
$\Aut(G_s)$. Let $\eta:G\to G_s$ be the natural projection and let $T'_t$ be the automorphism of $G_s$ corresponding to $T_t$ 
for each $t$. We now show that each $T'_t$ is trivial. Since $\{T'_t\}$ is unipotent and connected, there exists a $\Ad$-unipotent 
one-parameter subgroup $\{u_t\}$ in $G_s$ such that $T'_t=\inn (u_t)$, the inner automorphism by $u_t$ in $G_s$. Suppose 
$\{u_t\}$ is nontrivial.

Let $U$ be a maximal connected $\Ad$-unipotent subgroup of $G_s$ containing $\{u_t\}$, i.e.\ a maximal 
connected nilpotent subgroup containing $\{u_t\}$ such that for all $u\in U$, $\Ad(u)$ is a unipotent linear transformation of the Lie 
algebra of $G_s$. Let $G_s=KAN$ be an Iwasawa decomposition, where $A$ is a closed abelian subgroup, $N$ is a maximal 
connected $\Ad$-unipotent subgroup of $G_s$ which is normalised by $A$, and $AN$ is a closed solvable subgroup of $G_s$ 
(see the proof of Theorem 6.46 in \cite{Kn}). 
Since any two maximal connected Ad-unipotent subgroups are conjugate to each other, we get that $N$ and $U$ are conjugate 
by an element in a compact set contained in $K$ (see Lemma 4.4 in \cite{DS}). If necessary, we may replace $AN$ by its 
conjugate and assume that $U=N$, i.e.\ $u_t\in N$ for all $t\in\R$. Then $AN$ is $T'_t$-invariant. 

Let $B=\eta^{-1}(AN)$, where $\eta:G\to G/R$ is as above for the radical $R$ of $G$. Then $B$ is a closed connected solvable 
subgroup of $G$. Note that since $AN$ is $T'_t$-invariant, we have that $B$ is $T_t$-invariant for each $t$. 
Since $A$ and $N$ are simply connected and $R$ has no compact central subgroup of positive dimension, we get that $B$ 
has no compact central subgroup of positive dimension. As each $T_t|_{B}$ belongs to $\nc$, by Proposition \ref{solv}, 
$T_t|_{B}=\Id$ for each $t$. This implies that $u_t$ centralises $AN$ for all $t\in\R$. Let ${\mathcal G}_s$ be the Lie algebra of 
$G_s$ and let ${\mathcal A}$ and ${\mathcal N}$ be the Lie subalgebra of $A$ and $N$ respectively. 
If $X\in {\mathcal N}$ is such that $\exp tX=u_t$, $t\in\R$, then from above, $X$ belongs to the center of 
${\mathcal A}\oplus{\mathcal N}$, the Lie algebra of $AN$. This leads to a contradiction as the centraliser of ${\mathcal A}$ in 
${\mathcal A}\oplus{\mathcal N}$ is ${\mathcal A}$ \cite[Lemma 6.50]{Kn}. Therefore, $u_t=e$, and hence 
$T'_t=\Id$, for each $t$.

\noindent{\bf Step 3:} Now we have from Steps 1 and 2 that $T_t|_R=\Id$ and $T_t$ acts trivially on $G/R$ for each $t$, where $R$ is
the radical of $G$. If $R$ is central in $G$, then the preceding assertions would imply that each $T_t$ acts trivially on $\ol{[G,G]}$, the 
closure of the commutator subgroup of $G$, and as $G=\ol{[G,G]}R$, we get that $T_t=\Id$ in this case and, in particular, when $G$ 
is semisimple. Now we want to prove that $T_t=\Id$, for each $t$ in the general case.

For the natural projection $\eta:G\to G/R$, the subgroup $AN$ of $G/R$ and the subgroup $B=\eta^{-1}(AN)$ of $G$ as 
in Step 2, let $B_x=xBx^{-1}$, for a fixed $x\in G$. Then $B_x$ is a closed connected solvable subgroup of $G$, 
 $\eta(B_x)=\eta(x)AN\eta(x)^{-1}$ is also a closed connected solvable group of $G/R$, and $B_x$ is $T_t$-invariant, as $T_t$ acts 
 trivially on $G/R$, for each $t$. Since $B_x$ is a conjugate of $B$, it does not have any compact central subgroup of positive 
 dimension. Now by Proposition \ref{solv}, we get that the restriction of each $T_t$ to $B_x$ is trivial. Since this is true for every 
$x\in G$ and the closed subgroup generated by $H=\cup_{x\in G}\eta(B_x)$ is a closed normal subgroup containing $AN$ in $G/R$, 
we have that $H$ contains all the non-compact simple factors of $G/R$. Therefore, each $T_t$ is trivial on a closed co-compact 
normal subgroup $\eta^{-1}(H)$ in $G$ which contains $R$. That is, if $G/R$ has no nontrivial compact simple factors, then 
$T_t=\Id$, for each $t$. 

\noindent{\bf Step 4:} Let $S'$ be the product of all compact simple factors of $G/R$ and let $G'=\eta^{-1}(S')$. Then $G'$ is a closed 
characteristic (normal) subgroup in $G$. Let $G'=S_cR$ be the Levi decomposition of $G'$. Then $S_c$ is compact and it is the 
product of all simple compact factors of a Levi subgroup in $G$. As $G=G'\eta^{-1}(H)=S_c\eta^{-1}(H)$, and $T_t$ acts trivially on 
$\eta^{-1}(H)$, it is enough to show that $T_t|_{S_c}=\Id$. 

Here, $G'=S_cR$ is a closed subgroup. As $G'$ is characteristic, it is $T_t$-invariant for each $t$. Let $[S_c,R]$ be the group 
generated by $\{xyx^{-1}y^{-1}\mid x\in S_c, y\in R\}$. By 
Theorem 3.2 of Ch.\ XI of \cite{Ho}, $\ol{[S_c,R]}\subset N$; here $N$ denotes the nilradical of $R$, which is the same as the 
nilradical of $G$. Note that $\ol{[R,R]}\subset N$ \cite[Proposition 3.11]{Rag}. Let $G'_1$ be the closure of the group $[G',G']$. Then 
$S_c\subset G'_1\subset S_cN$. As $G'_1$ is $T_t$-invariant, we have that $T_t(S_c)\subset S_cN$. Since each $T_t$ acts trivially on 
$G'/R$, we have that for $x\in S_c$, $x^{-1}T_t(x)\in (S_c\cap R)N$. As $S_c\cap R$ is a finite subgroup and $S_c$ is connected,
 $T_t(x)\in xN$ for all $x\in S_c$ and for all $t$. Note that the nilradical $N$ of $G$ is simply connected and has no nontrivial compact 
 subgroups. Let $K$ be a maximal compact connected abelian subgroup of $S_c$. If possible, suppose for some $t$, $T_t|_K\ne\Id$. 
 Let $S_{T_t}$ be the stabiliser of $T_t$. Then $K\cap S_{T_t}$ is a proper closed subgroup of $K$. Now as $K/(K\cap S_{T_t})$ is 
 a nontrivial compact connected abelian group, it is monothetic and hence, there exists $x\in K$ such that the image of  $x$ in 
 $K/(K\cap S_{T_t})$ generats a dense infinite group. In particular, $G_x\cap S_{T_t}=\{e\}$. As $T_t(x)\in xN$, $T_t|_N=\Id$ and $N$ 
 has no nontrivial compact subgroups, by Lemma \ref{discrete}, $T_t\not\in\nc$, which leads to a contradiction. Therefore, $T_t|_K=\Id$ 
 for all $t$. Since the union of all maximal compact connected abelian subgroups of $S_c$ is dense in $S_c$, we have that $T_t|_{S_c}=\Id$ 
 for all $t$. As observed above, this shows that $T_t=\Id$ for all $t$. As observed at the end of Step 1, this implies that $T$ is contained in
 a compact subgroup of $G$. Therefore, \ $(i)\Rightarrow (iv)$ and $(i) - (iv)$ are equivalent. 
\end{proof}

We know from Remark \ref{rem1} that all automorphisms of connected compact abelian Lie groups of dimension greater than or equal to 2 
belong to class (NC) and, in particular, this hold for nontrivial unipotent automorphisms of such groups. The following theorem shows 
that a connected Lie group $G$ does not admit any nontrivial unipotent automorphism which acts distally on $\Sub^a_G$. 

\begin{thm} \label{unip} Let $G$ be a connected Lie group and let $T\in \Aut(G)$ be unipotent. Then $T$ acts distally 
on $\Sub^a_G$ if and only if $T=\Id$. 
\end{thm}

\begin{proof} 
Let $G$ and $T$ be as in the hypothesis.  One way implication is obvious. 
Now suppose $T$ acts distally on $\Sub^a_G$. We show that $T=\Id$.

Let $R$ be the solvable radical of $G$. Then $R$ is $T$-invariant and by Proposition \ref{solv}, $T|_R=\Id$. 
Let $C$ be the largest compact connected central subgroup. If $C$ is trivial, then the assertion follows from Theorem \ref{main} as the only 
unipotent automorphism contained in a compact subgroup of $\Aut(G)$ is the identity map. 
Suppose $C$ is nontrivial. Let $\T$ be the automorphisms of $G/C$ corresponding to $T$. Then by Lemma \ref{abelian}, 
$\T\in\nc$. Note that $\T$ is also unipotent. Now by Theorem \ref{main}, $\T$ is contained in a compact subgroup 
of $\Aut(G/C)$, and being unipotent, $\T$ is trivial. As $C$ is connected and central, it follows that
$T$ acts trivially on $\ol{[G,G]}$. Now as $G=\ol{[G,G]}R$, we get that $T=\Id$. \end{proof}

The following theorem shows that a part of Theorem \ref{main} can be generalised to all connected Lie groups if we restrict the class 
of automorphisms to $(\Aut(G))^0$, the connected component of the identity in $\Aut(G)$. Example \ref{ex} illustrates that 
Theorem \ref{main} can not be generalised fully even if we restrict to $(\Aut(G))^0$. 

\begin{thm} \label{aut0}
Let $G$ be a connected Lie group and let $T\in\Aut(G)$. Suppose $T$ belongs to $(\Aut(G))^0$, the connected component of 
the identity in $\Aut(G)$. Then $(ii)-(iv)$ in Theorem \ref{main} are equivalent. 
\end{thm}

\begin{proof} 
Let $G$ be a connected Lie group and let $T\in(\Aut(G))^0$. 
It is enough to show that $(ii)\Rightarrow (iv)$ in Theorem \ref{main}. Suppose $(ii)$ holds, 
i.e.\ $T$ acts distally on $\Sub^a_G$. We want to prove that $T$ is contained in a compact subgroup of $\Aut(G)$. 

Let $C$ be the largest compact connected central subgroup of $G$. Then $C$ is characteristic. As $C$ is compact and abelian, 
$\Aut(C)$ is totally disconnected and hence every element of $(\Aut(G))^0$ acts trivially on $C$. In particular, $T|_C=\Id$. 
Also, by Corollary \ref{Lie-distal}\,$(a)$, $T$ acts distally on $G$. Therefore, the eigenvalues of $\du T$ on the Lie algebra $\G$ 
of $G$ are of absolute value 1. Moreover, $(\Aut(G))^0$ is almost algebraic as a closed subgroup of $\GL(\G)$  \cite{W1, W2};
see also Theorem 2 in \cite{D}. Let $G_T$ be the smallest almost algebraic group in $(\Aut(G))^0$ containing $T$. By Corollary 2.5 
of \cite{A1}, $G_T=\K\times \U$ where $\K$ is compact and abelian, and $\U$ is unipotent and abelian. In fact, $\U$ is a unipotent 
one-parameter group $\{T_t\}$ in $(\Aut(G))^0$, such that $T=T_sT_u=T_uT_s$, $T_s\in\K$ and $T_u=T_1$ is unipotent. 
By Lemma 2.2 of \cite{SY2},  $T_u$ acts distally on $\Sub^a_G$. By Theorem \ref{unip}, $T_u=\Id$. Therefore, $T=T_s$ is 
contained in a compact subgroup of $(\Aut(G))^0$, i.e.\ $(iv)$ holds. 
\end{proof}

The following corollary characterises the class of connected Lie groups $G$ which act distally on $\Sub^a_G$. 
Recall that the action of $G$ on $\Sub_G$ is the same as the action of $\Inn(G)$ on $\Sub_G$, where $\Inn(G)$ is the group of 
inner automorphisms of $G$. 

\begin{cor} \label{pt-distal} Let $G$ be a connected Lie group. Then the following are equivalent.
\begin{enumerate}
\item[$(1)$] Every inner automorphism acts distally on $\Sub^a_G$.
\item[$(2)$] Every inner automorphism acts distally on $\Sub_G$. 
\item[$(3)$] $G$ acts distally on $\Sub^a_G$.
\item[$(4)$] $G$ acts distally on $\Sub_G$. 
\item[$(5)$] Either $G$ is compact or $G=\R^n\times K$, for a compact group $K$ and 
some $n\in\N$. 
\end{enumerate}
Moreover, if $G$ has no compact central subgroup of positive dimension, then the above statements are also equivalent to the following:
\begin{enumerate}
\item[$(6)$] Every inner automorphism of $G$ belongs to $\nc$. 
\end{enumerate}
\end{cor}

The condition on the center of $G$ in the above corollary is necessary for the equivalence of the statement $(6)$ with $(1-5)$ as illustrated 
by Example \ref{ex} below.

\begin{proof} [Proof of Corollary \ref{pt-distal}.] Let $G$ be a connected Lie group.
$(4)\Rightarrow (3)\Rightarrow (1)$ and $(4)\Rightarrow (2)\Rightarrow (1)$ obviously hold. We also have that $(5)\Rightarrow (4)$, as 
$\Inn(G)=\{\inn(k)\mid k\in K\}$ is compact. To show that $(1)-(5)$ are equivalent, it is enough to show that $(1)\Rightarrow (5)$. 
Suppose $(1)$ holds, i.e.\ every inner automorphism of $G$ acts distally on $\Sub^a_G$. We want to show that $(5)$ holds 

Suppose $G$ is not compact. We need to show that $G$ is a product of $\R^n$ and a compact subgroup 
for some $n\in\N$. As every inner automorphism acts distally on $\Sub^a_G$, by 
Corollary \ref{alm-conn}, $G$ is distal. By Theorem 9 of \cite{Ro} and Corollary 2.1\,(ii) of \cite{Je}, $G/R$ is compact where 
$R$ is the solvable radical of $G$. Let $N$ be the nilradical of $G$. Let $C$ be the largest compact normal subgroup of $N$. 
Then $C$ is connected and central in $G$ and $N/C$ is simply connected and nilpotent. Note that every inner automorphism of $G$ 
by an element of $N$ is unipotent, i.e.\ $\Ad(x)$ is unipotent for all $x\in N$. By Theorem \ref{unip}, $\inn(x)=\Id$ for every $x\in N$ 
i.e.\ $N$ is abelian and central in $G$. In particular, since $R/N$ is abelian, $R$ is nilpotent, and hence $R=N$. Therefore, $R$ is 
central in $G$. Now $G$ is a compact extension of $Z^0$, the connected component of the center $Z$ of $G$. Let $Z^0=\R^n\times C$, 
where $C$ is compact and, $n\in\N$ as $G$ is non-compact. Then the subgroup $V=\R^n$ is a (nontrivial) vector group which is central. 
Using Lemma 3.7 of \cite{I}, we get that $G=\R^n\times K$, where $K$ is the maximal compact subgroup of $G$, i.e.\ $(5)$ holds.

If $G$ does not have any compact central subgroup of positive dimension, by Theorem \ref{main}, we get that 
$(6) \Leftrightarrow (1)$, and hence the last assertion follows. 
\end{proof}

An automorphism $T$ of a locally compact group $G$ is said to have bounded orbits if all its orbits are relatively compact, i.e.\ for every 
$x\in G$, $\{T^{n}(x)\}_{n\in\Z}$ is relatively compact. In particular, if $T$ is contained in a compact subgroup of $\Aut(G)$, 
then $T$ has bounded orbits. The following result is for any automorphism on a general connected Lie group which acts 
distally on $\Sub^a_G$. 

\begin{cor} \label{Lie}
Let $G$ be a connected Lie group and let $T\in\Aut(G)$. If $T$ acts distally on $\Sub^a_G$, then $T$ has bounded orbits. 
\end{cor}

\begin{proof}
Suppose $T$ acts distally on $\Sub^a_G$. Let $C$ be the largest compact connected central subgroup of $G$. Then $C$ is characteristic 
in $G$ and $G/C$ does not have any compact central subgroup of positive dimension. Moreover, by Lemma \ref{abelian}, the automorphism 
$\T$ of $G/C$ corresponding to $T$ belongs to $\nc$ in $\Aut(G/C)$ and  by Theorem \ref{main}, we get that $\T$ has bounded orbits. 
Therefore, $T$ itself has bounded orbits, since $C$ is compact. 
\end{proof}

The converse of Corollary \ref{Lie} does not hold as illustrated by the following example. The example also illustrates that in 
Theorem \ref{unip}, the statement that a unipotent automorphism $T$ of a connected Lie group $G$  acts distally on $\Sub^a_G$ can 
not be replaced by  the statement that $T$ belongs to the class (NC). It also shows that Theorem \ref{main} can not be 
completely generalised to any connected Lie group $G$ even if we restrict the class of automorphisms to $(\Aut(G))^0$, as in 
Theorem \ref{aut0}. Moreover, the example also illustrates that the condition on the center in Corollary \ref{pt-distal} is necessary for 
the statement $(6)$ to be equivalent to $(1-5)$ in the theorem.

\begin{example} \label{ex}
Let $G=\HH/D$, where $\HH$ is the 3-dimensional Heisenberg group $($the group of strictly upper triangular - unipotent - matrices in 
$\SL(3,\R))$ and $D$ is a discrete central subgroup of $\HH$ isomorphic to $\Z$. Then $G$ is a connected step-2 nilpotent Lie group, 
its center $Z$ of $G$ is compact and $G/Z\cong\R^2$. Let $T$ be any nontrivial inner automorphism of $G$, i.e.\ $T=\inn(x)$ for some 
$x\not\in Z$, Note that $T$ has bounded orbits as $T^n(g)\in gZ$, $n\in\Z$, for every $g\in G$. As $G$ is connected and nilpotent, 
$T$ is nontrivial and unipotent and Theorem \ref{unip} implies that $T$ does not act distally on $\Sub^a_G$. However, 
$T$ belongs to (NC). For if $A$ is a discrete cyclic subgroup of $G$, it is either finite and hence central in $G$, in which case 
it is $T$-invariant, or its image in $G/Z$ is a discrete infinite cyclic group and as $T$ acts trivially on $G/Z$, it follows that $T\in\nc$. 
Also $T\in (\Aut(G))^0$ and in fact, it generates a non-compact subgroup in $(\Aut(G))^0$. Moreover, every inner automorphism of 
$G$ belongs to $\nc$. 
\end{example}

The map $T$ in the above example acts trivially on the central torus and it also acts trivially on the quotient group modulo the central torus. 
Such maps on connected Lie groups are known as (isotropic) shear automorphisms (see \cite{D}). 

\begin{rem} \label {rem2} $(1)$ For the unit circle $\Sc^1$, $\Aut(\Sc^1)$ is finite. Therefore, using Corollary 2.5 of \cite{PW}, 
Corollary \ref{Lie-distal} and Theorem \ref{aut0}, one can show that $(ii)-(iv)$ of Theorem \ref{main} are equivalent for a somewhat 
larger class of connected Lie groups whose maximal torus is central and it is of dimension at most 1. This class includes connected 
nilpotent Lie groups whose largest compact (connected) central subgroup is of dimension 1, e.g.\ $G$ as in Example \ref{ex}. 
It also includes the Lie group of the form $\SL(2,\R)\ltimes G$, where $G$ is as in the above example. $(2)$ It would be interesting to 
not only identify the spaces $\Sub^a_G$ and $\Sub_G$, but describe the action of $\Aut(G)$ on them, as it is done in \cite{BHK} for 
some subspaces of $\Sub_\HH$, for the 3-dimensional Heisenberg group $\HH$. $(3)$ In case of compact abelian $($Lie$)$ groups $G$, 
we study the action of unipotent automorphisms. It would be interesting to study the action of a general $T$ in $\Aut(G)$ on 
$\Sub_G$ for such compact groups $G$, as it would give more insight for the action of $\Aut(G)$ on $\Sub_G$ for a general connected 
Lie group $G$. $(4)$ For general locally compact groups $G$, $($especially for totally disconnected groups$)$, a comprehensive study 
is needed for  the action of $\Aut(G)$ on $\Sub_G$. 
\end{rem}

We conclude with a mention of the following: Using many results in this paper, further study of class (NC) and the distality of actions of
automorphisms of $G$ on $\Sub_G$ have been carried out where $G$ is a certain type of locally compact group or a lattice in a
Lie group \cite{PaS, PPS}. Expansivity of the action of automorphisms of $G$ on $\Sub_G$ has also been studied for 
locally compact groups and lattices in Lie groups recently  \cite{PrS, PPS}.

\medskip
\noindent{\bf Acknowledgements} R.\ Shah would like to thank S.\ G.\ Dani for extensive discussions on the structure of semisimple 
Lie groups. R.\ Shah would also like to thank I.\ Chatterji and Universit\'e Nice Sophia Antipolis, France, for hospitality during 
a visit in June 2019 while a part of work was carried out. R.\ Shah would like to acknowledge the MATRICS research grant from 
DST-SERB, Govt.\ of India, which partially supported A.\ K.\ Yadav as a visiting scholar at JNU, New Delhi in 2018. A.\ K.\ Yadav 
would like to thank the Ambedkar University Delhi for an adjunct faculty position in 2018 and Krishnendu Gongopadhyay for a 
visiting position under a research grant from NBHM, DAE, Govt.\ of India, at IISER Mohali in early 2019. We thank the 
anonymous referee for suggestions which led to some improvements in the presentation of the manuscript.

\bigskip\medskip
\noindent{
\advance\baselineskip by 2pt
\begin{tabular}{ll}
Riddhi Shah & \hspace*{1cm}Alok Kumar Yadav \\
School of Physical Sciences(SPS)& \hspace*{1cm}Department of Mathematical Sciences\\
Jawaharlal Nehru University(JNU) & \hspace*{1cm}Indian Institute of Science Education - \\
New Delhi 110 067, India& \hspace*{1cm}and Research (IISER Mohali)\\
rshah@jnu.ac.in& \hspace*{1cm}Mohali 140306, India\\
riddhi.kausti@gmail.com& \hspace*{1cm}alokmath1729@gmail.com
\end{tabular}}
\end{document}